\documentclass[a4paper,10pt]{amsart}
\usepackage[utf8]{inputenc}
\usepackage{amsthm}
\usepackage{amsmath}
\usepackage{bm}
\usepackage{amsfonts}
\usepackage{amssymb}
\usepackage{mathtools}
\usepackage{mathrsfs}
\usepackage{setspace}
\usepackage{xcolor}
\usepackage{textgreek}
\usepackage{todonotes}
\usepackage[a4paper, left=2cm, right=2cm, top=2cm, bottom=3cm]{geometry}
\usepackage{thmtools} % solves problem of shared counters in statements and autoref

\usepackage[pdfdisplaydoctitle,colorlinks,breaklinks,urlcolor=blue,linkcolor=blue,citecolor=blue]{hyperref} %must always be the last

\newtheorem{thm}{Theorem}[section]

\newtheorem{cor}[thm]{Corollary}

\newtheorem{lem}[thm]{Lemma}

\newtheorem{prop}[thm]{Proposition}

\newcommand\numberthis{\addtocounter{equation}{1}\tag{\theequation}}
\theoremstyle{remark}
\newtheorem{rmk}[thm]{Remark}

\newcommand{\R}{\mathbb{R}}
\newcommand{\E}{\mathbb{E}}
\newcommand{\N}{\mathbb{N}}

\newcommand{\PP}{\mathbb{P}}

\newcommand{\ag}{\left\{}
\newcommand{\cg}{\right\}}

\definecolor{myGreen}{rgb}{0.0,0.5,0.0}

\title[Stochastic model reduction]{Stochastic model reduction: convergence \\ and applications to climate equations}
\author[S. Assing]{Sigurd Assing}
  \address{Department of Statistics, The University of Warwick, Coventry CV4 7AL, UK}
  \email{\href{mailto:s.assing@warwick.ac.uk}{s.assing@warwick.ac.uk}}
\author[F. Flandoli]{Franco Flandoli}
  \address{Scuola Normale Superiore, Piazza dei Cavalieri, 7, 56126 Pisa, Italia}
  \email{\href{mailto:franco.flandoli@sns.it}{franco.flandoli@sns.it}}
\author[U. Pappalettera]{Umberto Pappalettera}
  \address{Scuola Normale Superiore, Piazza dei Cavalieri, 7, 56126 Pisa, Italia}
  \email{\href{mailto:umberto.pappalettera@sns.it}{umberto.pappalettera@sns.it}}
\date\today

\begin{document}

\begin{abstract}
We study stochastic model reduction for evolution equations in infinite dimensional Hilbert spaces,
and show the convergence to the reduced  equations via abstract results of Wong-Zakai type for stochastic equations 
driven by a scaled Ornstein-Uhlenbeck process.
Both weak and strong convergence are investigated, depending on the presence of quadratic interactions 
between reduced variables and driving noise.
Finally, we are able to apply our results to a class of equations used in climate modeling.
\end{abstract}

\maketitle

%%%%%%%%%%%%%%%%%%%%%%%%%%%%%%%%%%%%%%%%%%%%%%%%%%%%%%%%%%%%%%%%%%%%%%%%%%%%%%%%%%%%%%%%%%%%%

\section{Introduction} \label{sec:intro}
In this paper we study stochastic model reduction for a system of nonlinear evolution equations in infinite dimensional Hilbert spaces
which is general enough to cover well-established systems of equations used in climate modeling.
The big advantage of such a procedure is the lower complexity of the reduced equations,
since complexity is still one of the major issues 
when predicting the evolution of systems over time spans which are typical for climate rather than meteorology.

Following \cite{MaTiVE01}, %where finite-dimensional systems are considered, 
we assume that the \emph{climate variables} of the system, i.e.\ those more relevant to climate prediction, evolve on longer times scales than the \emph{unresolved variables}, which can be modelled stochastically and have a typical time scale much shorter than the climate variables.
%Similar to \cite{MaTiVE01},
To be able to close the equation for the climate variables,
the task is to understand the effects of unresolved variables when stretching time to climate-time.
In what follows, we also refer to climate variables as \emph{resolved variables}.

Climate modeling typically starts with equations containing quadratic nonlinearities 
which can describe many features of oceanic and atmospheric dynamics at meteorological time---see \cite{MaWa06,Va06}.
In abstract mathematical terms, such equations would look like
\begin{equation} \label{eq:introZ}
 \frac{dZ_t}{dt} = f_t + A Z_t + B(Z_t,Z_t),
\end{equation} 
where $A:H \to H$ is a linear operator, $B:H\times H \to H$ is a bilinear operator, and $f$ is an external forcing term.
Here, the variable $Z$ taking values in $H$ is supposed to be a complex mix of climate and unresolved variables,
and hence the space $H$ has to be `big enough' to `host' variables of that type.
We therefore choose $H$ to be a separable infinite-dimensional Hilbert space.

Now, there is a variety of procedures to identify climate variables in practice which we will not discuss in this paper.
We rather assume that climate variables have been identified spanning a Hilbert-subspace $H_d\subset H$,
and we further assume that the orthogonal complement $H_\infty,\,H = H_d \oplus H_\infty$, 
gives the space of unresolved variables.
When projecting $Z$ onto $H_d$, $H_\infty$ via the projection maps $\pi_d$, $\pi_\infty$,
equation \eqref{eq:introZ} gives raise to two equations
\begin{equation}
\frac{dX_t}{dt} = f^1_t + \tilde{A}^1_1 X_t + A^1_2 Y_t + \tilde{B}^1_{11}(X_t,X_t) + B^1_{12}(X_t,Y_t) + B^1_{22}(Y_t,Y_t) \label{eq:introX}
\end{equation}
and
\begin{equation}
\frac{dY_t}{dt} = f^2_t + A^2_1 X_t + A^2_2 Y_t + B^2_{11}(X_t,X_t) + B^2_{12}(X_t,Y_t) + B^2_{22}(Y_t,Y_t) \label{eq:introY}
\end{equation}
for the collection of climate variables $X=\pi_d(Z)$ and unresolved variables $Y=\pi_\infty(Z)$, respectively.

The next step,
called stochastic climate modeling, 
consists in replacing the complicated nonlinear self-interaction term in \eqref{eq:introY} by a linear random term.
Such a replacement could be justified by the assumption that
quickly varying fluctuations of small scale unresolved variables are more or less indistinguishable from the 
combined effect of a large number of weakly coupled factors, 
usually leading to Gaussian driving forces via Central Limit Theorem.
But such effects would only become visible at climate time 
and not at meteorological time used in \eqref{eq:introX} \&\ \eqref{eq:introY},
so that we are looking to replace $B^2_{22}(Y_{\varepsilon^{-1}t},Y_{\varepsilon^{-1}t})$
by a linear random term, stretching meteorological time to $\varepsilon^{-1}t$,
using a small parameter $\varepsilon \ll 1$. 

In this work, following \cite{MaTiVE01,PeMa94}, we suppose that
%\begin{equation} \label{eq:subs1}
\[
B^2_{22}(Y_{\varepsilon^{-1}t},Y_{\varepsilon^{-1}t}) \mbox{ is replaced by } 
- \mu \varepsilon^{-1} Y_{\varepsilon^{-1}t} + \sigma \dot{W}_t,
\]
%\end{equation}
where $\mu,\sigma$ are positive constants, 
and $\dot{W}$ is Gaussian noise, white in time, and coloured in space.
This way, the parameter $\varepsilon$ is used to scale time,
but also to adjust for the size of the involved variables when scaling time.
%This choice is motivated both by the statistical properties of the stochastic forcing \cite{PeMa94}, 
%and by the explicit computability of the Ornstein-Uhlenbeck process.

Another assumption made in \cite{MaTiVE01}
is that climate variables at climate time have small forcing and self-interaction, and hence we also suppose that
%\begin{equation} \label{eq:subs2}
\[
f^1_{\varepsilon^{-1}t} + \tilde{A}^1_1 X_{\varepsilon^{-1}t} + \tilde{B}^1_{11}(X_{\varepsilon^{-1}t},X_{\varepsilon^{-1}t}) 
\mbox{ is replaced by }
\varepsilon F^1_t + \varepsilon A^1_1 X_{\varepsilon^{-1}t} + \varepsilon B^1_{11}(X_{\varepsilon^{-1}t},X_{\varepsilon^{-1}t}),
\]
%\end{equation}
avoiding so-called fast forcing and fast waves.

All in all, 
when introducing the notation $X^\varepsilon_t = X_{\varepsilon^{-1}t}$ for climate variables at climate time,
and $Y^\varepsilon_t = \varepsilon^{-1} Y_{\varepsilon^{-1}t}$ for the effect of unresolved variables at climate time,
equations \eqref{eq:introX} \&\ \eqref{eq:introY} translate into
\begin{align}
\frac{dX^\varepsilon_t}{dt} &= F^1_t +  A^1_1 X^\varepsilon_t +  A^1_2 Y^\varepsilon_t + B^1_{11}(X^\varepsilon_t,X^\varepsilon_t) + B^1_{12}(X^\varepsilon_t,Y^\varepsilon_t) + \varepsilon B^1_{22}(Y^\varepsilon_t,Y^\varepsilon_t), \label{eq:introXeps}\\
\frac{dY^\varepsilon_t}{dt} &= \varepsilon^{-2} f^2_{\varepsilon^{-1} t} 
+ \varepsilon^{-2} A^2_1 X^\varepsilon_t 
+ \varepsilon^{-1}A^2_2 Y^\varepsilon_t 
+ \varepsilon^{-2} B^2_{11}(X^\varepsilon_t,X^\varepsilon_t) 
+\varepsilon^{-1} B^2_{12}(X^\varepsilon_t,Y^\varepsilon_t) 
-\mu \varepsilon^{-2}Y^\varepsilon_t + \sigma \varepsilon^{-2} \dot{W}_t . \label{eq:introYeps}
\end{align}

The hope is now that, when $\varepsilon$ tends to zero, climate variables at climate time can be approximated
by a random variable $\bar{X}$ which solves a closed stochastic equation with new coefficients not depending on
unresolved variables any more.
Of course, these new coefficients will be functions of the coefficients of equations \eqref{eq:introXeps} \&\ \eqref{eq:introYeps},
and the process of finding these new coefficients is called {\em stochastic model reduction}.

Stochastic model reduction of finite-dimensional systems similar to \eqref{eq:introXeps},\eqref{eq:introYeps}
were extensively discussed in \cite{MaTiVE01}.
However, one of the key steps, i.e.\ proving the convergence $X^\varepsilon\to\bar{X},\,\varepsilon\downarrow 0$, was kept rather short.
Indeed, the authors first sketch a perturbation method based on a theorem by T.G.\ Kurtz, \cite{Kurtz73}, which is their general method,
and they then briefly describe a so-called {\em direct averaging method} for special cases
based on limits of solutions to stochastic differential equations. In particular the latter method lacks a certain amount of rigour 
because the convergence of the involved stochastic processes is not shown,
and this gap has not been closed in follow-up papers---see \cite{FrMaVa05,FrMa06,JaTiVa15} for example.

In this paper we are not only closing this gap, but also develop a new method of proof.

We at first identify $\bar{X}$, and then study in very detail the convergence 
$X^\varepsilon\to\bar{X},\,\varepsilon\downarrow 0$,
when $X^\varepsilon$ solves an evolution equation of type
\begin{equation} \label{eq:introabsX}
\frac{dX^\varepsilon_t}{dt} = 
F(t,X^\varepsilon_t) + 
\sigma(t,X^\varepsilon_t) Y^\varepsilon_t 
+\varepsilon \beta(Y^\varepsilon_t,Y^\varepsilon_t),
\end{equation}
where $Y^\varepsilon$ is a decoupled infinite-dimensional Ornstein-Uhlenbeck process satisfying
\begin{equation} \label{eq:introabsY}
\frac{dY^\varepsilon_t}{dt} =
%{\color{red}\varepsilon^{-1}A^2_2 Y^\varepsilon_t}
-\varepsilon^{-2}Y^\varepsilon_t 
+\varepsilon^{-2} \dot{W}_t.
\end{equation}

Since equation \eqref{eq:introabsX} is more general than \eqref{eq:introXeps}, 
once stochastic model reduction is established for the system \eqref{eq:introabsX},\eqref{eq:introabsY}
with decoupled unresolved variables,
it also follows for an interesting subclass of systems of type \eqref{eq:introXeps},\eqref{eq:introYeps}
with coupled unresolved variables---see \autoref{thm:climate}. %and \autoref{thm:climate1}.
%Indeed, for the latter to hold true,
%one only has to restrict \eqref{eq:introYeps} to equations with
%$f^2_{\varepsilon^{-1} t}=0$ (that is no fast forcing of unresolved variables) and $B^2_{12}=0$, 
%as \autoref{thm:climate} and \autoref{thm:climate1} show.
Part (ii) of this theorem deals with the case of linear scattering, that is $B^1_{22} =0$,
and in this case we achieve showing `strong' convergence in probability:
\begin{equation} \label{eq:introlim}
\lim_{\varepsilon \to 0} \PP\ag \sup_{t \leq T }\|X^\varepsilon_t - \bar{X}_t\|_{H_d} > \delta\cg = 0,
\quad \forall \delta>0,
\end{equation}  
on a given climate time interval $[0,T]$. When the quadratic interaction term $B^1_{22}$ is non-trivial,
we can only show convergence in law, as stated in \autoref{thm:climate}(i).
We refer to \autoref{rmk:conv}(ii)   for an argument which suggests that 
one cannot expect much more than a weak-type convergence in the general case. 
This insight of course sheds new light on the results given in \cite{MaTiVE01} and follow-up papers.

%{\color{red}
At this point it should be mentioned that thoughout this paper we assume that $H_d$ is finite-dimensional
which seems to be a natural choice when it comes to climate modeling.
However, our arguments are general 
and can be adapted to infinite dimensional subspaces, see \cite{FlPa20+}. 
%}

In the case of the more abstract system \eqref{eq:introabsX},\eqref{eq:introabsY},
the process $Y^\varepsilon$ will eventually behave like white noise, as $\varepsilon \downarrow 0$.
This limiting behaviour is fundamental for finding the limit of equation \eqref{eq:introabsX}
because it opens the door for using arguments similar to those of Wong \&\ Zakai in \cite{WoZa65}.
Of course, Wong \&\ Zakai formulated their results in a finite-dimensional setting.
There have been earlier attempts of proving similar results in infinite dimensions,
we refer to \cite{BrCaFl88,Tw93,TeZa06}, for example.
However, we would like to emphasise that these earlier attempts dealt with
piecewise linear approximations of noise rather than an infinite dimensional Ornstein-Uhlenbeck process.
Note that it is typical for Wong-Zakai results that
stochastic integral terms of limiting equations are interpreted in the sense of Stratonovich.

The paper is structured as follows. 

In \autoref{sec:not}, 
we formulate our main results on the convergence of solutions to \eqref{eq:introabsX},\eqref{eq:introabsY}.
First, the limiting equation for $\bar{X}$ is identified, 
and then conditions for weak convergence $X^\varepsilon\to\bar{X}$ are stated in \autoref{thm:main}(i).
However, when \eqref{eq:introabsX} is a simpler equation, i.e.\ $\beta=0$,
even the stronger convergence \eqref{eq:introlim} can be shown 
under the same conditions---see \autoref{thm:main}(ii).

In \autoref{sec:main},
we give the proof of \autoref{thm:main}(ii). 
The proof relies on preliminary localization and discretization arguments which allow to consider, 
instead of \eqref{eq:introlim}, its discrete version
\begin{equation*} 
\lim_{\varepsilon \to 0} \PP\ag \sup_{k}\|X^\varepsilon_{t_k} - \bar{X}_{t_k}\|_{H_d} > \delta\cg = 0,
\quad \forall \delta>0,
\end{equation*} 
for only finitely many $t_k \in [0,T]$. 

In \autoref{sec:weakWZ}, 
we give the proof of \autoref{thm:main}(i) which, at the beginning,
requires a careful analysis of the quadratic term $\beta(Y^\varepsilon_t,Y^\varepsilon_t)$,
but otherwise is an adaptation of the proof given in the previous section.

In \autoref{sec:clim},
we eventually use the results of \autoref{sec:not} to prove \autoref{thm:climate} %and \autoref{thm:climate1}
under quite natural conditions,
thus making the connection to our main applications in climate modeling.
\section{Notation and main Result} \label{sec:not}
%
%In this section we illustrate our main convergence results in the abstract setting.
%\subsection{Strong convergence: a Wong-Zakai result}
%
Let $H_d$, $H_\infty$ be real separable Hilbert spaces.
Assume that $H_d$ is finite-dimensional, $\dim H_d = d$, with given orthonormal basis 
$\mathbf{e}_1,\dots,\mathbf{e}_d$,
and that $H_\infty$ is infinite-dimensional with given orthonormal basis
$\mathbf{f}_1,\mathbf{f}_2,\dots$

Given two Banach spaces $U,V$, let $\mathcal{L}(U,V)$ denote the Banach space of continuous linear operators 
mapping $U$ to $V$, endowed with the operator norm.

For each $\varepsilon>0$, consider the pair of stochastic processes $(X^\varepsilon,Y^\varepsilon)$,
taking values in $H_d \times H_\infty$,
where  $X^\varepsilon$ satisfies \eqref{eq:introabsX} over a fixed finite time interval $[0,T]$,
and $Y^\varepsilon$ is given by
%Consider the system of equations
%\begin{equation} \label{eq:Xeps}
%X^\varepsilon_t = x_0 
%+ \int_0^t \left(\rule{0pt}{12pt}
%F(s,X^\varepsilon_s)
%+ \sigma(s,X^\varepsilon_s) Y^\varepsilon_s 
%+ \varepsilon \beta(Y^\varepsilon_t,Y^\varepsilon_t)
%\right) ds, \quad
\[
Y^\varepsilon_t = \int_{-\infty}^t \varepsilon^{-2}e^{-\varepsilon^{-2}(t-s)} dW_s,\quad t\ge 0,
\]
%\end{equation}
%for the pair of processes $(X^\varepsilon,Y^\varepsilon)\in H_d \times H_\infty$
%over a fixed finite time interval $[0,T]$.
where $W$ is a Wiener process in $H_\infty$, 
with real-valued time parameter
and
self-adjoint trace class covariance operator $Q \in \mathcal{L}(H_\infty,H_\infty)$.
%given on a probability space $(\Omega,\mathcal{F},\PP)$.
\begin{rmk}\label{realTimeWipro}\rm
(i)
A Wiener process with real-valued time parameter can be obtained in the following way: 
given two independent Wiener processes $(W^+_t)_{t\ge 0}$ and $(W^-_t)_{t\ge 0}$
defined on filtered probability spaces
 $(\Omega^+,(\mathcal{F}^+_t),\PP^+)$ 
and $(\Omega^-,(\mathcal{F}^-_t),\PP^-)$, respectively,
set $W_t = W^+_t$, for $t \geq 0$, and $W_t = W^-_{-t}$, for $t<0$. 

(ii) 
Using such a representation of $W$, we can also write
\begin{equation*}
Y^\varepsilon_t = 
- \int_0^\infty \varepsilon^{-2}e^{-\varepsilon^{-2}(t+s)} dW^-_s
+ \int_0^t \varepsilon^{-2}e^{-\varepsilon^{-2}(t-s)} dW^+_s,
\quad t\ge 0,
\end{equation*}
which clearly is a \underline{stationary} Ornstein-Uhlenbeck process on 
$(\Omega,\mathcal{F}^-_\infty \otimes \mathcal{F}^+_\infty,\PP)$,
where $\Omega=\Omega^- \times \Omega^+$ and $\PP=\PP^- \otimes \PP^+$, see \cite{DPZa14}. 
Furthermore, setting up the stochastic basis for our processes $(X^\varepsilon,Y^\varepsilon)$,
let $(\Omega,\mathcal{F},\PP)$ be the completion of 
$(\Omega,\mathcal{F}^-_\infty \otimes \mathcal{F}^+_\infty,\PP)$,
and $(\mathcal{F}_t)_{t \geq 0}$ be the augmentation of the filtration
$(\mathcal{F}^-_\infty \otimes \mathcal{F}^+_t)_{t \geq 0}$.
Note that this filtration would satisfy the usual conditions.

(iii)
Since $Q$ is trace class, both $W$ and $Y^\varepsilon$  take values in $H_\infty$.
Without loss of generality, we can assume that $Q$ is diagonal with respect to 
the chosen basis $\{\mathbf{f}_m\}_{m \in \N}$ of $H_\infty$,
that the eigenvalues of $Q$ form a sequence $\{q_m\}_{m \in \N}$ satisfying
$\sum_{m} q_m < \infty$,
and that $\E \left[\langle W_t, \mathbf{f}_m \rangle_{H_\infty} ^2\right] = |t| q_m$, for all $m$.
\end{rmk}

Adopting the useful notation $W^\varepsilon_t = \int_0^t Y^\varepsilon_s ds$, 
we can write \eqref{eq:introabsX} in integral form as
\begin{equation} \label{eq:Xeps'}
X^\varepsilon_t = x_0 
+ \int_0^t  %\left(\rule{0pt}{12pt}
F(s,X^\varepsilon_s) ds   %+ \varepsilon\beta(Y^\varepsilon_s,Y^\varepsilon_s)\right) ds
+ \int_0^t \sigma(s,X^\varepsilon_s) dW^\varepsilon_s
+ \int_0^t  \varepsilon\beta(Y^\varepsilon_s,Y^\varepsilon_s) ds,
\quad t\in[0,T],
\end{equation}
where $x_0 \in H_d$ is a deterministic initial condition, as well as
$F:[0,T] \times H_d \to H_d$, 
$\sigma:[0,T] \times H_d \to \mathcal{L}(H_\infty,H_d)$,
$\beta:H_{\infty} \times H_{\infty} \to H_d$.
We make the following assumptions on these coefficients:
\begin{itemize}
\item[(\textbf{A1})] $F \in C([0,T] \times H_d , H_d)$, 
and $F(t,\cdot) \in {Lip}_{loc}(H_d,H_d)$, uniformly in $t \in [0,T]$;
\item[(\textbf{A2})] $\sigma \in C^1([0,T] \times H_d,\mathcal{L}(H_\infty,H_d))$, 
and its space-differential $D\sigma(t,\cdot) \in {Lip}_{loc}(H_d,\mathcal{L}(H_d,\mathcal{L}(H_\infty,H_d)))$, 
uniformly in $t \in [0,T]$;
\item[(\textbf{A3})] $\beta:H_{\infty} \times H_{\infty} \to H_d$ is  a continuous bilinear map.
\end{itemize}

Of course, by standard theory (see \cite{DPZa14} for example),
equation \eqref{eq:Xeps'} admits a unique local strong solution, for each $\varepsilon >0$.

Next, we introduce the limiting equation for the wanted limit $\bar{X}$ of the processes $X^\varepsilon$,
when $\varepsilon\downarrow 0$.
First, define the so-called \emph{Stratonovich correction} term $C:[0,T] \times H_d \to H_d$ by
\begin{equation} \label{eq:Strat}
C^i(s,{x}) = 
\langle C(s,{x}) ,   \mathbf{e}_i \rangle_{H_d}
= \frac12 \sum_{m \in \N} q_m \sum_{j=1}^d D_j \sigma^{i,m}(s,{x}) \sigma^{j,m}(s,{x}), 
\quad i=1,\dots,d,
\end{equation}
where
\begin{align*}
\sigma^{i,m}(s,x) = \langle\sigma(s,x)\mathbf{f}_m , \mathbf{e}_i \rangle_{H_d},
\quad  i=1,\dots,d,\;m\in\N,
\end{align*}
is matrix notation for the linear map $\sigma(s,x)\in\mathcal{L}(H_\infty,H_d)$
with respect to our chosen basis vectors;
second, let
\begin{equation}\label{diffCoeff}
b^i_{\ell,m} = \langle\beta(\mathbf{f}_\ell,\mathbf{f}_m) , \mathbf{e}_i \rangle_{H_d}\,
\sqrt{\frac{q_\ell q_m}{2}},
\quad  i=1,\dots,d,\;\ell,m\in\N.
\end{equation}
Then, our limiting equation would read
\begin{equation} \label{eq:X}
\bar{X}_t = x_0 
+ \int_0^t \left(\rule{0pt}{12pt}
F(s,\bar{X}_s) + C(s,\bar{X}_s)
\right) ds 
+ \int_0^t \sigma(s,\bar{X}_s ) d{W}_s
+ \sum_{\ell, m \in \N} b_{\ell,m} \bar{W}^{\ell,m}_t,
\quad t\in[0,T],
\end{equation}
where $W$ is the same Wiener process used to define $Y^\varepsilon$ in \autoref{realTimeWipro},
while $\{\bar{W}^{\ell,m}\}_{\ell,m \in \N}$ is a family of independent 
one-dimensional standard Wiener processes,
which are also independent of $W$.

Again by standard theory, this equation admits a unique local strong solution, too.
However, 
in view of the interpretation of our results with respect to climate modeling,
it is natural to further assume that
\begin{itemize}
\item[(\textbf{A4})] both equations \eqref{eq:Xeps'} and \eqref{eq:X} 
admit global solutions on  $[0,T]$.
\end{itemize}

Another assumption specific to climate modeling,
which has been advocated in \cite{MaTiVE01}, for example, would be that the mean of
$\beta(Y^\varepsilon_s,Y^\varepsilon_s)$ is zero, for any $s$,
with respect to the invariant measure of the corresponding Ornstein-Uhlenbeck process.
Since all $Y^\varepsilon$ are stationary under $\PP$, see \autoref{realTimeWipro}(ii),
this assumption would translate into
\begin{align*}
\E \left[ \langle\beta(Y^\varepsilon_s,Y^\varepsilon_s) , \mathbf{e}_i \rangle_{H_d}
\right] &=
 \sum_{\ell,m \in \N} 
 \langle\beta(\mathbf{f}_\ell,\mathbf{f}_m) , \mathbf{e}_i \rangle_{H_d}\,
 \E \left[  Y^{\varepsilon,\ell}_s Y^{\varepsilon,m}_s \right]  
\,=
\sum_{\ell \in \N} 
 \langle\beta(\mathbf{f}_\ell,\mathbf{f}_\ell) , \mathbf{e}_i \rangle_{H_d}\,
 \frac{\varepsilon^{-2}}{2}  q_\ell
 \,=\,0,
\end{align*}
where $Y_s^{\varepsilon,\ell}$ is short notation
for the coordinates $\langle Y_s^{\varepsilon} , \mathbf{f}_\ell \rangle_{H_\infty}$,
$\ell=1,2,\dots,\,s\in[0,T]$.
As a consequence, we also impose the zero-mean condition
\begin{itemize}
\item[(\textbf{A5})] 
$\sum_{\ell \in \N} 
 \langle\beta(\mathbf{f}_\ell,\mathbf{f}_\ell) , \mathbf{e}_i \rangle_{H_d}\, q_\ell\,=0$,
 for all $i=1,\dots,d$,
\end{itemize}
which is usually true for equations from fluid-dynamics
and can in general be understood as a renormalization procedure for the quadratic term.

The following theorem is the main result of this paper.
\begin{thm} \label{thm:main}
(i)
Assume (A1)-(A5). Then, $X^\varepsilon$ converges to $\bar{X}$, in law, $\varepsilon\downarrow 0$.

(ii)
However, if (A1)-(A4) and (A5) comes via $\beta=0$, then the stronger convergence \eqref{eq:introlim} holds true.
\end{thm}

In what follows, 
to keep notation light in proofs, when no confusion may occur,
the norms in both spaces $H_d$ and $H_\infty$ will be denoted by $|\cdot|$, 
and their scalar products by $\langle \cdot,\cdot\rangle$.
The symbol $\lesssim$ means inequality up to a multiplicative constant,
possibly depending on the parameters of our equations, but not on $\varepsilon$. 
\section{Strong convergence} \label{sec:main}
In this section we give the proof of \autoref{thm:main}(ii), 
which is divided into several steps.

First, by localization, we argue that we can restrict ourselves
to $|X^\varepsilon_t|$, $|\bar{X}_t| \leq R$, for some large $R$,
which is effectively leading to Lipschitz continuity of the coefficients of \eqref{eq:Xeps'}. 

Second, we discretize the problem, which allows us to reduce the proof of \autoref{thm:main}(ii) to its discrete version:
\begin{equation*} 
\lim_{\varepsilon \to 0} \PP\ag \sup_{k}|X^\varepsilon_{t_k} - \bar{X}_{t_k}| > \delta\cg = 0,
\quad \forall \delta>0,
\end{equation*} 
for only finitely many $t_k \in [0,T]$. 
Here, we choose $t_k = k\Delta$, where $\Delta=\Delta_\varepsilon$ is a positive parameter 
whose $\varepsilon$-dependence has to be carefully chosen in the proof---see \autoref{choiceDelta}.

Third, we prove the above discretized version.  
\subsection{Localization} \label{subsec:loc}
Fix $\varepsilon>0,\,\delta\in(0,1)$, and define 
\[
\tau^\varepsilon_R = \inf\{t\geq 0: |X^\varepsilon_t| \geq R+1\} \wedge \inf\{t\geq 0: |\bar{X}_t| \geq R\},
\quad\mbox{for $R>0$},
\]
so that
\begin{align*}
\PP\ag \sup_{t \leq T }|X^\varepsilon_t - \bar{X}_t| > \delta\cg 
&= 
\PP\ag \sup_{t \leq T }|X^\varepsilon_t - \bar{X}_t| > \delta, \, \sup_{t \leq T } |\bar{X}_t| \geq R \cg 
+ \PP\ag \sup_{t \leq T }|X^\varepsilon_t - \bar{X}_t| > \delta, \, \sup_{t \leq T } |\bar{X}_t| < R \cg \\
&= \PP\ag \sup_{t \leq T }|X^\varepsilon_t - \bar{X}_t| > \delta, \, \sup_{t \leq T } |\bar{X}_t| \geq R \cg
+ \PP\ag \sup_{t \leq T \wedge \tau^\varepsilon_R }|X^\varepsilon_t - \bar{X}_t| > \delta, \, \sup_{t \leq T } |\bar{X}_t| < R \cg \\
&\leq \PP\ag \sup_{t \leq T } |\bar{X}_t| \geq R \cg + \PP\ag \sup_{t \leq T \wedge \tau^\varepsilon_R }|X^\varepsilon_t - \bar{X}_t| > \delta\cg. \numberthis \label{goodSummand}
\end{align*}

Therefore, since (A4) implies
\begin{equation*}
\PP \ag \sup_{t \leq T } |\bar{X}_t| \geq R \cg \to  0, \mbox{ as } R \uparrow \infty,
\end{equation*}
to prove \eqref{eq:introlim}, 
it is sufficient to show the convergence of the second summand on the right-hand side of \eqref{goodSummand},
when $\varepsilon\downarrow 0$, for fixed $\delta\in(0,1),\,R>0$.
Furthermore, by Markov inequality,
\begin{equation} \label{eq:markov}
\PP\ag\sup_{t \leq T\wedge \tau^\varepsilon_R }|X^\varepsilon_t - \bar{X}_t| > \delta \cg \leq \delta^{-p}\, \E \left[ \sup_{t \leq T\wedge\tau^\varepsilon_R }|X^\varepsilon_t - \bar{X}_t|^p\right],
\end{equation}
for every $p>0$, $\delta \in (0,1)$,
and hence showing convergence of the above right-hand side, only, is enough.
To keep  notation light, 
we are going to use $\tau^\varepsilon$ instead of  $\tau^\varepsilon_R$,
as $R>0$ will be fixed, in what follows.
\subsection{Discretization}
Fix $\varepsilon>0$.
We show that the expectation on the right-hand side of \eqref{eq:markov} can be replaced by an expectation of the same quantity, but with the supremum taken over a finite number 
(diverging to $\infty$, as $\varepsilon \downarrow 0$) of times $t_k$, see \autoref{cor:discr} below.

To start with, we have the following useful a priori estimate.
\begin{lem} \label{lem:ou}
For any $p>1$, the Ornstein-Uhlenbeck process $Y^\varepsilon$ satisfies
\begin{equation*}
\E \left[ \sup_{t \leq T}\left| Y^\varepsilon_t \right|^p  \right] \lesssim\varepsilon^{-p}.
\end{equation*}
\end{lem}
\begin{proof}
First, the result is true in one dimension---see \cite[Theorem 2.2]{JiZh20}. 

In the infinite dimensional case, by H\"older's inequality, we can suppose $p>2$. 
Therefore, since $Q$ is trace class with eigenvalues satisfying $\sum_{m\in\N}q_m<\infty$,
when $\alpha = (p-2)/p$, we obtain that
\begin{align*}
\E \left[ \sup_{t \leq T}\left| Y^\varepsilon_t \right|^p  \right] 
&= 
\E \left[ \sup_{t \leq T} \left( \sum_{m \in \N, q_m>0} q_m^{\alpha} q_m^{-\alpha} \left| Y^{\varepsilon,m}_t \right|^2 \right)^{p/2} \right] \\
&\lesssim 
\left( \sum_{m \in \N, q_m>0}
q_m^{-\alpha p/2} \E \left[ \sup_{t \leq T}\left| Y^{\varepsilon,m}_t \right|^p  \right] \right)
\left( \sum_{m \in \N}
q_m^{\alpha p/(p-2)} \right)^{(p-2)/2}
\lesssim
\varepsilon^{-p},
\end{align*}
having used the one-dimensional result for the coordinates
$Y^{\varepsilon,m}_t = \langle Y^\varepsilon_t , \mathbf{f}_m \rangle,\,m=1,2,\dots$
\end{proof}

Now, we introduce the discretization of the time interval $[0,T]$. 
Let $\Delta>0$, and let  $[T/\Delta]$ be the largest integer less or equal than $T/\Delta$. 
In what follows, $\Delta$ will also depend on $\varepsilon$, in a way to be determined later. 
Also, to make it easier to bound terms by powers of $\varepsilon$ or $\Delta$,
without loss of generality, we will always assume that both $\varepsilon,\Delta$ are less than \underline{one}.

The next two lemmas control the excursion of $X^\varepsilon$ between adjacent nodes
in terms of the ratio $\Delta/\varepsilon$.
\begin{lem}\label{lem:discrXeps}
For any  $p>1$, and any deterministic time $\tau>0$,
\begin{equation*}
\E \left[ \sup_{\substack{k=0,1,\dots,[T/\Delta] \\t \leq \tau ,\, t+k\Delta \leq T\wedge\tau^\varepsilon }} |{X}^\varepsilon_{t+k\Delta} - {X}^\varepsilon_{k\Delta}|^p \right] \lesssim \left( \frac{\tau}{\varepsilon}\right)^p.
\end{equation*}
\end{lem}
\begin{proof}
Since $\beta=0$, by \eqref{eq:Xeps'},
 the increment ${X}^\varepsilon_{t+k\Delta} - {X}^\varepsilon_{k\Delta}$ can be written as
\begin{align*}
{X}^\varepsilon_{t+k\Delta} - {X}^\varepsilon_{k\Delta} 
= &\int_{k\Delta}^{t+k\Delta} F(s,X^\varepsilon_s) ds + \int_{k\Delta}^{t+k\Delta} \sigma(s,X^\varepsilon_s)dW^\varepsilon_s,
\quad\mbox{for $t+k\Delta\leq T\wedge\tau^\varepsilon$}.
\end{align*}
Therefore, using (A1),(A2), boundedness of $X^\varepsilon$ on $[0,\tau^\varepsilon]$, and \autoref{lem:ou},
we obtain that
\begin{align*}
\E \left[ \sup_{\substack{k=0,1,\dots,[T/\Delta] \\t \leq \tau ,\, t+k\Delta \leq T\wedge\tau^\varepsilon }} |{X}^\varepsilon_{t+k\Delta} - {X}^\varepsilon_{k\Delta}|^p \right]  
&\lesssim
 \tau^p \left( 1 + \E \left[ \sup_{t \leq T\wedge\tau^\varepsilon }\left| Y^\varepsilon_t \right|^p  \right] \right) 
\lesssim
 \left( \frac{\tau}{\varepsilon}\right)^p,
\end{align*}
where $W^\varepsilon_t = \int_0^t Y^\varepsilon_s ds$ was defined in \autoref{sec:not}.
\end{proof}
\begin{lem}\label{lem:discrXepsbis}
For any $p>1$, and any fixed $k \in \{0,1,\dots,[T/\Delta]\}$ such that $k\Delta\le T$, 
\begin{equation*}
\E \left[ |{X}^\varepsilon_{(k+1)\Delta\wedge\tau^\varepsilon} - {X}^\varepsilon_{k\Delta\wedge\tau^\varepsilon}|^p \right] \lesssim 
\Delta^{p/2} + \varepsilon^{p} + \left( \frac{\Delta}{\varepsilon}\right)^{2p}.
\end{equation*}
\end{lem}
\begin{proof}
It suffices to bound every single term on the right-hand side of the equation
\begin{align*}
X^\varepsilon_{(k+1)\Delta\wedge\tau^\varepsilon} - X^\varepsilon_{k\Delta\wedge\tau^\varepsilon} 
= &\int_{k\Delta\wedge\tau^\varepsilon}^{(k+1)\Delta\wedge\tau^\varepsilon} F(s,X^\varepsilon_s) ds 
\\
&+
\int_{k\Delta\wedge\tau^\varepsilon}^{(k+1)\Delta\wedge\tau^\varepsilon} \left( \sigma(s,X^\varepsilon_s) - \sigma(k\Delta\wedge\tau^\varepsilon,X^\varepsilon_{k\Delta\wedge\tau^\varepsilon}) \right) dW^\varepsilon_s  
\\
&+ 
\int_{k\Delta\wedge\tau^\varepsilon}^{(k+1)\Delta\wedge\tau^\varepsilon} \sigma(k\Delta\wedge\tau^\varepsilon,X^\varepsilon_{k\Delta\wedge\tau^\varepsilon}) dW^\varepsilon_s.
\end{align*}
First, by (A1) and boundedness of $X^\varepsilon$ on $[0,\tau^\varepsilon]$,
we have that
\begin{align*}
\E \left[ \left| \int_{k\Delta\wedge\tau^\varepsilon}^{(k+1)\Delta\wedge\tau^\varepsilon} F(s,X^\varepsilon_s) ds \right|^p \right] 
&\lesssim \Delta^p.
\end{align*}
Second, using H\"older's inequality with $q'>1$, (A2), \autoref{lem:ou} and \autoref{lem:discrXeps},
\begin{align*}
\E &\left[ \left| \int_{k\Delta\wedge\tau^\varepsilon}^{(k+1)\Delta\wedge\tau^\varepsilon} \left( \sigma(s,X^\varepsilon_s) - \sigma(k\Delta\wedge\tau^\varepsilon,X^\varepsilon_{k\Delta\wedge\tau^\varepsilon}) \right) dW^\varepsilon_s  \right|^p \right] 
\\
&\lesssim 
\E \left[ \sup_{t \leq T}\left| Y^\varepsilon_t \right|^p \left| \int_{k\Delta\wedge\tau^\varepsilon}^{(k+1)\Delta\wedge\tau^\varepsilon} \left| \sigma(s,X^\varepsilon_s) - \sigma(k\Delta\wedge\tau^\varepsilon,X^\varepsilon_{k\Delta\wedge\tau^\varepsilon}) \right| ds  \right|^p \right] \\
&\lesssim 
\varepsilon^{-p} \,
\left(\rule{0pt}{13pt}\right.
\E \left[ \left| \int_{k\Delta\wedge\tau^\varepsilon}^{(k+1)\Delta\wedge\tau^\varepsilon} \left| \sigma(s,X^\varepsilon_s) - \sigma(k\Delta\wedge\tau^\varepsilon,X^\varepsilon_{k\Delta\wedge\tau^\varepsilon}) \right| ds  \right|^{pq'} \right]
\left.\rule{0pt}{13pt}\right)^{1/q'} \\
&\lesssim 
\varepsilon^{-p} \Delta^{p-1/q'} \left( \int_{k\Delta\wedge\tau^\varepsilon}^{(k+1)\Delta\wedge\tau^\varepsilon} \E \left[ \left| X^\varepsilon_s - X^\varepsilon_{k\Delta\wedge\tau^\varepsilon}  \right|^{pq'} + (s-k\Delta)^{pq'} \right] ds \right)^{1/q'} \\
&\lesssim \left(\frac{\Delta}{\varepsilon}\right)^{2p}.
\end{align*}
Finally, 
\begin{align*}
\E \left[ \left| \int_{k\Delta\wedge\tau^\varepsilon}^{(k+1)\Delta\wedge\tau^\varepsilon} \sigma(k\Delta\wedge\tau^\varepsilon,X^\varepsilon_{k\Delta\wedge\tau^\varepsilon}) dW^\varepsilon_s \right|^p \right] 
&\lesssim 
\E \left[ \left|  W^\varepsilon_{(k+1)\Delta\wedge\tau^\varepsilon} - W^\varepsilon_{k\Delta\wedge\tau^\varepsilon}  \right|^p \right] \lesssim
\Delta^{p/2} + \varepsilon^{p},
\end{align*}
because, for every $t_2 > t_1 \geq 0$,
\begin{align} \label{eq:Weps}
W^\varepsilon_{t_2} - W^\varepsilon_{t_1}=
&\int_{t_1}^{t_2} \left(\int_{-\infty}^s \varepsilon^{-2}e^{-\varepsilon^{-2}(s-r)}dW_r\right)ds \\
= &W_{t_2} - W_{t_1}
- \int_{-\infty}^{t_2} e^{-\varepsilon^{-2}({t_2}-r)}dW_r 
+ \int_{-\infty}^{t_1} e^{-\varepsilon^{-2}({t_1}-r)}dW_r. \nonumber
\end{align}
\end{proof}

The next lemma controls the excursion of the limiting process $\bar{X}$ between adjacent nodes.
\begin{lem}\label{lem:discrXbis}
For any $p>1$, any deterministic time $\tau\in(0,1)$,
and any fixed $k \in \{0,1,\dots,[T/\Delta]\}$,
\begin{equation*}
\E \left[ \sup_{t \leq \tau,\,t+k\Delta \leq T\wedge\tau^\varepsilon}|\bar{X}_{t+k\Delta} - \bar{X}_{k\Delta}|^p \right] \lesssim \tau^{\frac{p}{2}}.
\end{equation*}
\end{lem}
\begin{proof}
Since $\beta=0$, by \eqref{eq:X}, 
the increment $\bar{X}_{t+k\Delta} - \bar{X}_{k\Delta}$ can be written as
\begin{align*}
\bar{X}_{t+k\Delta} - \bar{X}_{k\Delta} = \int_{k\Delta}^{t+k\Delta} \left( F(s,\bar{X}_s) + C(s,\bar{X}_s)\right) ds + \int_{k\Delta}^{t+k\Delta} \sigma(s,\bar{X}_s ) dW_s,
\quad\mbox{for $t+k\Delta\leq T\wedge\tau^\varepsilon$}.
\end{align*}
Therefore, using (A1),(A2), boundedness of $X^\varepsilon$ on $[0,\tau^\varepsilon]$,
and Burkholder-Davis-Gundy's inequality, we obtain that
\begin{align*}
\E \left[ \sup_{t \leq \tau,\,t+k\Delta \leq T\wedge\tau^\varepsilon}|\bar{X}_{t+k\Delta} - \bar{X}_{k\Delta}|^p \right]
&\lesssim \tau^p + \E \left[ \sup_{t \leq \tau,\,t+k\Delta \leq T\wedge\tau^\varepsilon} \left| \int_{k\Delta}^{t+k\Delta}  
\sigma(s,\bar{X}_s ) dW_s \right|^p  \right] \lesssim \tau^p + \tau^{\frac{p}{2}},
\end{align*}
which proves the lemma since $\tau<1$.
\end{proof}
\begin{cor}\label{lem:discrX}
For any $p>1$,
\begin{equation*}
\E \left[ \sup_{\substack{k=0,1,\dots,[T/\Delta] \\t \leq \Delta ,\, t+k\Delta \leq T\wedge\tau^\varepsilon }} 
|\bar{X}_{t+k\Delta} - \bar{X}_{k\Delta}|^p \right] \lesssim \Delta^{\frac{p}{2}-1}.
\end{equation*}
\end{cor}
\begin{proof}
The claim easily follows from \autoref{lem:discrXbis} with $\tau=\Delta$, and the inequality
\begin{equation*}
\E \left[ \sup_{\substack{k=0,1,\dots,[T/\Delta] \\t \leq \Delta ,\, t+k\Delta \leq T\wedge\tau^\varepsilon }} 
|\bar{X}_{t+k\Delta} - \bar{X}_{k\Delta}|^p \right] 
\lesssim 
\sum_{k=0}^{[T/\Delta]} \E \left[ \sup_{t \leq \Delta ,\, t+k\Delta \leq T\wedge\tau^\varepsilon } 
|\bar{X}_{t+k\Delta} - \bar{X}_{k\Delta}|^p \right]. 
\end{equation*}
\end{proof}
\begin{cor} \label{cor:discr}
Let $\Delta=\Delta_\varepsilon>0$ depend on $\varepsilon$ such that
$\Delta/\varepsilon \to 0$, as $\varepsilon \downarrow 0$. Then,
\begin{equation*}
\E \left[ \sup_{t \leq T\wedge\tau^\varepsilon }|X^\varepsilon_t - \bar{X}_t|^2\right]
\lesssim
\E \left[ \sup_{\substack{k=0,1,\dots,[T/\Delta] \\ k\Delta \leq \tau^\varepsilon}} |X^\varepsilon_{k\Delta} - \bar{X}_{k\Delta}|^2\right] + o(1).
\end{equation*}
\end{cor}
\begin{proof}
First, by H\"older's inequality with $q>1$ and \autoref{lem:discrX}, we have that
\begin{equation*}
\E \left[ \sup_{\substack{k=0,1,\dots,[T/\Delta] \\t \leq \Delta ,\, t+k\Delta \leq T\wedge\tau^\varepsilon }} 
|\bar{X}_{t+k\Delta} - \bar{X}_{k\Delta}|^2 \right] 
\lesssim 
\left(\rule{0pt}{13pt}\right.
\E \left[ \sup_{\substack{k=0,1,\dots,[T/\Delta] \\t \leq \Delta ,\, t+k\Delta \leq T\wedge\tau^\varepsilon }} 
|\bar{X}_{t+k\Delta} - \bar{X}_{k\Delta}|^{2q} \right]
\left.\rule{0pt}{13pt}\right)^{1/q} 
\lesssim \Delta^{1-1/q}.
\end{equation*}
Thus, the proof can easily be completed by combining the above and \autoref{lem:discrXeps},
while taking into account
\[
X^\varepsilon_t - \bar{X}_t
\,=\,
(X^\varepsilon_t - X^\varepsilon_{[t/\Delta]\Delta})
+(X^\varepsilon_{[t/\Delta]\Delta}-\bar{X}_{[t/\Delta]\Delta})
+(\bar{X}_{[t/\Delta]\Delta}-\bar{X}_t),
\]
where $[t/\Delta]$ is again our notation for the floor of $t/\Delta$.
\end{proof}
\subsection{Proof of the discretized version}
By \eqref{eq:markov} and \autoref{cor:discr}, it suffices to prove
\begin{equation} \label{eq:Esup2}
\E \left[ \sup_{\substack{k=0,\dots,[T/\Delta] \\  k\Delta \leq \tau^\varepsilon}} \left| X^\varepsilon_{k\Delta}-\bar{X}_{k\Delta}\right|^2\right] \to 0,
\quad\varepsilon\downarrow 0,
\end{equation}
for some $\Delta = \Delta_\varepsilon = o(\varepsilon)$. 
The proof is inspired by \cite[Section VI.7]{IkWa14}.

To start with,
by \eqref{eq:Xeps'} without $\beta$-term,  (A2), and \eqref{eq:Weps}, we have that
\begin{align*}
X^\varepsilon_{(k+1)\Delta} \numberthis \label{eq:incXeps}
= \,&X^\varepsilon_{k\Delta} 
+ \int_{k\Delta}^{(k+1)\Delta} F(s,X^\varepsilon_s) ds 
+ \int_{k\Delta}^{(k+1)\Delta} \sigma(s,X^\varepsilon_s) dW^\varepsilon_s \\
= \,&X^\varepsilon_{k\Delta} 
+ \int_{k\Delta}^{(k+1)\Delta} \left(F(s,X^\varepsilon_s)- F(k\Delta,X^\varepsilon_{k\Delta}) \right) ds \\
&+ \int_{k\Delta}^{(k+1)\Delta} F(k\Delta,X^\varepsilon_{k\Delta}) ds \\
&+ \int_{k\Delta}^{(k+1)\Delta} \left( \sigma(s,X^\varepsilon_s) -\sigma(k\Delta,X^\varepsilon_{k\Delta}) \right) dW^\varepsilon_s 
+ \int_{k\Delta}^{(k+1)\Delta}  \sigma(k\Delta,X^\varepsilon_{k\Delta}) dW^\varepsilon_s \\
= \,&X^\varepsilon_{k\Delta} 
+ \int_{k\Delta}^{(k+1)\Delta} \left(F(s,X^\varepsilon_s)- F(k\Delta,X^\varepsilon_{k\Delta}) \right) ds \\
&+ \int_{k\Delta}^{(k+1)\Delta} F(k\Delta,X^\varepsilon_{k\Delta}) ds \\
&+ \int_{k\Delta}^{(k+1)\Delta} \left( \int_{k\Delta}^s \left( \partial_r \sigma(r,X^\varepsilon_r) + D\sigma(r,X^\varepsilon_r)F(r,X^\varepsilon_r) \right) dr \right) dW^\varepsilon_s \\
&+ \int_{k\Delta}^{(k+1)\Delta} \left( \int_{k\Delta}^s \left( D\sigma(r,X^\varepsilon_r)\sigma(r,X^\varepsilon_r) - D\sigma(k\Delta,X^\varepsilon_{k\Delta})\sigma(k\Delta,X^\varepsilon_{k\Delta})\right) dW^\varepsilon_r \right) dW^\varepsilon_s \\
&+ \int_{k\Delta}^{(k+1)\Delta} \left( \int_{k\Delta}^s \left(  D\sigma(k\Delta,X^\varepsilon_{k\Delta})\sigma(k\Delta,X^\varepsilon_{k\Delta}) - D\sigma(k\Delta,\bar{X}_{k\Delta})\sigma(k\Delta,\bar{X}_{k\Delta}) \right) dW^\varepsilon_r \right) dW^\varepsilon_s \\
&+ \int_{k\Delta}^{(k+1)\Delta} \left( \int_{k\Delta}^s  D\sigma(k\Delta,\bar{X}_{k\Delta})\sigma(k\Delta,\bar{X}_{k\Delta})  dW^\varepsilon_r \right) dW^\varepsilon_s \\
&+ \int_{k\Delta}^{(k+1)\Delta}  \sigma(k\Delta,X^\varepsilon_{k\Delta}) dW_s \\
&+ \sigma(k\Delta,X^\varepsilon_{k\Delta}) \varepsilon^2 \left( Y^\varepsilon_{k\Delta} - Y^\varepsilon_{(k+1)\Delta} \right) \\
= \,&X^\varepsilon_{k\Delta} + I^k_1 + I^k_2 + I^k_3 + I^k_4 + I^k_5 + I^k_6 + I^k_7 + I^k_8,
\end{align*}
for any $k=0,\dots,[T/\Delta]$ such that $(k+1)\Delta\le T$.

Similarly, using \eqref{eq:X} instead of \eqref{eq:Xeps'},
the process $\bar{X}$ satisfies
\begin{align*}
\bar{X}_{(k+1)\Delta} \numberthis \label{eq:incX}
= \,&\bar{X}_{k\Delta} 
+ \int_{k\Delta}^{(k+1)\Delta} \left(F(s,\bar{X}_s)- F(k\Delta,\bar{X}_{k\Delta})\right) ds \\
&+ \int_{k\Delta}^{(k+1)\Delta} F(k\Delta,\bar{X}_{k\Delta}) ds \\
&+ \int_{k\Delta}^{(k+1)\Delta} \left( C(s,\bar{X}_s) - C(k\Delta,\bar{X}_{k\Delta})\right) ds \\
&+ \int_{k\Delta}^{(k+1)\Delta} C(k\Delta,\bar{X}_{k\Delta}) ds \\ 
&+ \int_{k\Delta}^{(k+1)\Delta} \left( \sigma(s,\bar{X}_s) - \sigma(k\Delta,\bar{X}_{k\Delta}) \right) dW_s + \int_{k\Delta}^{(k+1)\Delta}  \sigma(k\Delta,\bar{X}_{k\Delta}) dW_s \\
= \,&\bar{X}_{k\Delta} + J^k_1 + J^k_2 + J^k_3 + J^k_4 + J^k_5 + J^k_6. 
\end{align*}

Having in mind to apply Gronwall's lemma, it turns out to be useful to summarise the contributions of
the right-hand sides of \eqref{eq:incXeps}, \eqref{eq:incX} as follows:
\begin{align*}
X^\varepsilon_{h\Delta} - \bar{X}_{h\Delta}    \numberthis\label{contribute}
= &\sum_{k=0}^{h-1} \left( I^k_2 - J^k_2 \right)
 + \sum_{k=0}^{h-1} \left( I^k_6 - J^k_4 \right) 
 + \sum_{k=0}^{h-1} \left( I^k_7 - J^k_6 \right) 
 + \sum_{k=0}^{h-1} I^k_5 \\
&+ \sum_{k=0}^{h-1} \left( I^k_1 + I^k_3 + I^k_4 + I^k_8 - J^k_1 - J^k_3 - J^k_5 \right), \nonumber
\end{align*}
for any $h=1,\dots,[T/\Delta]$,
which splits the difference $X^\varepsilon_{h\Delta} - \bar{X}_{h\Delta}$ into 5 sums.

We at first prove that the 2nd and the 5th sum can be neglected when proving \eqref{eq:Esup2}. 
The summands of the 5th sum are discussed in \autoref{lem:aux} below.
The contribution of the 2nd sum though is more delicate and requires a martingale argument 
similar to that of \cite[Theorem VI.7.1]{IkWa14}.

The remaining sums will be controlled in terms of the difference $X^\varepsilon - \bar{X}$ itself,
which allows them to be estimated via Gronwall's lemma.

Of course, under assumption (A1), the function $F$ is uniformly continuous when restricted to  $[0,T] \times B_R(0)$, where $B_R(0)$ is the closed ball of radius $R$ in $H_d$. 
In what follows, we will denote by $\omega_F:[0,T] \to [0,\infty)$ the (local) modulus of continuity of $F(\cdot,x)$:
\begin{equation*}
\left| F(t,x) - F(s,x) \right| \leq \omega_F(|t-s|), \quad \mbox{ for every } t,s \in [0,T], \mbox{ and } x \in B_R(0).
\end{equation*}

Obviously, the function $\omega_F$ vanishes at zero, and without loss of generality, it can be chosen to be both non-decreasing and continuous. 
%taking for instance $\tilde{\omega}_F(t) = t^{-1} \int_t^{2t} \sup_{r\leq s} \omega_F(r\wedge T) ds $. 

Denote by $\omega_\sigma$ the corresponding modulus of continuity of the derivative $D\sigma(\cdot,x)$, 
and let $\omega_{F,\sigma} = \omega_F + \omega_\sigma$.

\begin{lem} \label{lem:aux}
For any $p>1${\rm :}
\begin{align*}
\E \left[ \sup_{\substack{h=1,\dots,[T/\Delta]\\h\Delta \leq \tau^\varepsilon}} \left| 
\sum_{k=0}^{h-1} I^k_1 \right|^p + \left| 
\sum_{k=0}^{h-1} I^k_3 \right|^p\right] &\lesssim \left( \frac{\Delta}{\varepsilon} \right)^p + \omega_F(\Delta)^p;\\
\E \left[ \sup_{\substack{h=1,\dots,[T/\Delta]\\h\Delta \leq \tau^\varepsilon}} \left| 
\sum_{k=0}^{h-1} I^k_4 \right|^p\right] &\lesssim \left( \frac{\Delta^2}{\varepsilon^3} \right)^p + \left( \frac{\Delta}{\varepsilon^2}\right)^p \omega_\sigma(\Delta)^p;\\
\E \left[ \sup_{\substack{h=1,\dots,[T/\Delta]\\h\Delta \leq \tau^\varepsilon}} \left| 
\sum_{k=0}^{h-1} I^k_8 \right|^p\right] &\lesssim \left( \frac{\varepsilon^2}{\Delta}\right)^{p/2} +\left( \frac{\varepsilon^2}{\Delta}\right)^{p} + \left( \frac{\Delta}{\varepsilon}\right)^{p};\\
\E \left[ \sup_{\substack{h=1,\dots,[T/\Delta]\\h\Delta \leq \tau^\varepsilon}} \left| 
\sum_{k=0}^{h-1} J^k_1 \right|^p + \left| 
\sum_{k=0}^{h-1} J^k_3 \right|^p + \left| 
\sum_{k=0}^{h-1} J^k_5 \right|^p\right] &\lesssim \Delta^{p/2} + \omega_{F,\sigma}(\Delta)^p.
\end{align*}
\end{lem}
\begin{proof}
Throughout this proof, we will frequently make use of (A1),(A2) without explicit mentioning.

For $\sum  I^k_1$, by H\"older's inequality and \autoref{lem:discrXeps},
\begin{align*}
%
%%%%%%%%%%%%%%%%%%%%% I_1
%
\E \left[ \sup_{\substack{h=1,\dots,[T/\Delta]\\h\Delta \leq \tau^\varepsilon}} \left| 
\sum_{k=0}^{h-1} I^k_1 \right|^p\right] &\lesssim
\E \left[ \sup_{\substack{h=1,\dots,[T/\Delta]\\h\Delta \leq \tau^\varepsilon}} \left| 
\sum_{k=0}^{h-1} \int_{k\Delta}^{(k+1)\Delta} \left(\left|X^\varepsilon_s - X^\varepsilon_{k\Delta} \right| + \omega_F(s-k\Delta) \right) ds  \right|^p\right] \\
&\lesssim
\sum_{k=0}^{[T/\Delta]-1} \int_{k\Delta}^{(k+1)\Delta } \E \left[  \left|X^\varepsilon_{s\wedge \tau^\varepsilon} - X^\varepsilon_{k\Delta \wedge \tau^\varepsilon} \right|^p +\omega_F(\Delta)^p \right] ds \\
&\lesssim \left( \frac{\Delta}{\varepsilon} \right)^p + \omega_F(\Delta)^p.
\end{align*}

For $\sum I^k_3$, by H\"older's inequality and \autoref{lem:ou},
\begin{align*}
%
%%%%%%%%%%%%%%%%%%%%%%%% I_3
%
\E \left[ \sup_{\substack{h=1,\dots,[T/\Delta]\\h\Delta \leq \tau^\varepsilon}} \left| 
\sum_{k=0}^{h-1} I^k_3 \right|^p\right] 
&\lesssim \E \left[ \sup_{\substack{h=1,\dots,[T/\Delta]\\h\Delta \leq \tau^\varepsilon}} \left|\, \sup_{t \leq T}\left| Y^\varepsilon_t \right|  
\sum_{k=0}^{h-1} \int_{k\Delta}^{(k+1)\Delta} \left( s-k\Delta\right) ds  \right|^p\right] \\
&\lesssim \E \left[ \sup_{t \leq T}\left| Y^\varepsilon_t \right|^p  
\sum_{k=0}^{[T/\Delta]-1} \int_{k\Delta}^{(k+1)\Delta} \left| s-k\Delta\right|^p ds \right] \lesssim \left( \frac{\Delta}{\varepsilon} \right)^p. 
\end{align*}

For $\sum  I^k_4$, by H\"older's inequality, \autoref{lem:ou} and \autoref{lem:discrXeps},
\begin{align*}
%
%%%%%%%%%%%%%%%%%%%%%%%% I_4
%
\E& \left[ \sup_{\substack{h=1,\dots,[T/\Delta]\\h\Delta \leq \tau^\varepsilon}} \left| 
\sum_{k=0}^{h-1} I^k_4 \right|^p\right]\\
&\lesssim
\E \left[ \sup_{\substack{h=1,\dots,[T/\Delta]\\h\Delta \leq \tau^\varepsilon}} \left|\, \sup_{t \leq T}\left| Y^\varepsilon_t \right|^2  
\sum_{k=0}^{h-1} \int_{k\Delta}^{(k+1)\Delta} \left( \int_{k\Delta}^s \left( \left| X^\varepsilon_r - X^\varepsilon_{k\Delta} \right| + \omega_\sigma(r-k\Delta) \right)  dr \right) ds  \right|^p\right] \\
&\lesssim
\E \left[ \sup_{\substack{h=1,\dots,[T/\Delta]\\h\Delta \leq \tau^\varepsilon}}  \sup_{t \leq T}\left| Y^\varepsilon_t \right|^{2p}  
\sum_{k=0}^{h-1} \int_{k\Delta}^{(k+1)\Delta} \left| \int_{k\Delta}^s \left( \left| X^\varepsilon_r - X^\varepsilon_{k\Delta} \right| + \omega_\sigma(r-k\Delta) \right) dr \right|^p ds  \right] \\ 
&\lesssim
\varepsilon^{-2p} \left(
\sum_{k=0}^{[T/\Delta]-1} \int_{k\Delta}^{(k+1)\Delta} (s-k\Delta)^{pq'-1}  
 \int_{k\Delta}^s \left( \E \left[ \left| X^\varepsilon_{r\wedge \tau^\varepsilon} - X^\varepsilon_{k\Delta \wedge \tau^\varepsilon} \right|^{pq'} + \omega_\sigma(\Delta)^{pq'} \right] dr \right) ds
 \right)^{1/q'}\\
 &\lesssim
\varepsilon^{-3p} \left( \sum_{k=0}^{[T/\Delta]-1}\int_{k\Delta}^{(k+1)\Delta} (s-k\Delta)^{2pq'}  ds
\right)^{1/q'} 
+ \left( \frac{\Delta}{\varepsilon^2}\right)^p \omega_\sigma(\Delta)^p \\
&\lesssim \left( \frac{\Delta^2}{\varepsilon^3} \right)^p + \left( \frac{\Delta}{\varepsilon^2}\right)^p \omega_\sigma(\Delta)^p.
\end{align*}

We now consider $\sum I^k_8$. 
Here, the idea is to convert $Y^\varepsilon$-increments into $X^\varepsilon$-increments
via integration by parts since $X^\varepsilon$-increments are easier to control. 
This way, applying \autoref{lem:ou} and \autoref{lem:discrXepsbis},
\begin{align*}
%
%%%%%%%%%%%%%%%%%%%%%%%%%%%%%% I_8
%
\E \left[ \sup_{\substack{h=1,\dots,[T/\Delta]\\h\Delta \leq \tau^\varepsilon}} \left| 
\sum_{k=0}^{h-1} I^k_8 \right|^p\right] 
&\lesssim
\E \left[ \sup_{\substack{h=1,\dots,[T/\Delta]\\h\Delta \leq \tau^\varepsilon}} \left| 
\sum_{k=0}^{h-1} \sigma(k\Delta,X^\varepsilon_{k\Delta})\varepsilon^2 \left( Y^\varepsilon_{k\Delta} - Y^\varepsilon_{(k+1)\Delta} \right) \right|^p\right]\\
&\lesssim
\E \left[ \sup_{\substack{h=1,\dots,[T/\Delta]\\h\Delta \leq \tau^\varepsilon}} \left| 
\sum_{k=1}^{h} \left( \sigma(k\Delta,X^\varepsilon_{k\Delta}) - \sigma((k-1)\Delta,X^\varepsilon_{(k-1)\Delta}) \right)\varepsilon^2 Y^\varepsilon_{k\Delta}\right|^p\right] \\
&\lesssim
\E \left[ \sup_{\substack{h=1,\dots,[T/\Delta]\\h\Delta \leq \tau^\varepsilon}} \sup_{t \leq T}\left| \varepsilon^2 Y^\varepsilon_t \right|^{p}  \left| 
\sum_{k=1}^{h} \left(\left| X^\varepsilon_{k\Delta} - X^\varepsilon_{(k-1)\Delta}\right| + \Delta\right)\right|^p\right] \\
&\lesssim
\E \left[ \sup_{t \leq T}\left| \varepsilon^2 Y^\varepsilon_t \right|^{pq} \right]^{1/q} 
\E \left[ \sup_{\substack{h=1,\dots,[T/\Delta]\\h\Delta \leq \tau^\varepsilon}} \left| 
\sum_{k=1}^{h} \left(\left| X^\varepsilon_{k\Delta} - X^\varepsilon_{(k-1)\Delta}\right| + \Delta\right) \right|^{pq'} \right]^{1/q'} \\
&\lesssim
\varepsilon^p \Delta^{1/q'-p}
\left( \sum_{k=1}^{[T/\Delta]} \E \left[  \left| X^\varepsilon_{k\Delta \wedge\tau^\varepsilon} - X^\varepsilon_{(k-1)\Delta\wedge\tau^\varepsilon}\right|^{pq'} + \Delta^{pq'} \right] \right)^{1/q'} \\
&\lesssim
\varepsilon^p \Delta^{-p}
\left( \Delta^{pq'/2} + \varepsilon^{pq'} + \left( \frac{\Delta}{\varepsilon}\right)^{2pq'} \right)^{1/q'} \\
&\lesssim 
\left( \frac{\varepsilon^2}{\Delta}\right)^{p/2} +\left( \frac{\varepsilon^2}{\Delta}\right)^{p} + \left( \frac{\Delta}{\varepsilon}\right)^{p}.
\end{align*}

In a similar way,
for $\sum J^k_1$ and $\sum J^k_3$, now applying \autoref{lem:discrXbis},
\begin{align*}
%
%%%%%%%%%%%%%%%%%%%%% J_1
%
\E \left[ \sup_{\substack{h=1,\dots,[T/\Delta]\\h\Delta \leq \tau^\varepsilon}} \left| 
\sum_{k=0}^{h-1} J^k_1 \right|^p + \left| 
\sum_{k=0}^{h-1} J^k_3 \right|^p\right] &\lesssim
\E \left[ \sup_{\substack{h=1,\dots,[T/\Delta]\\h\Delta \leq \tau^\varepsilon}} \left| 
\sum_{k=0}^{h-1} \int_{k\Delta}^{(k+1)\Delta} \left(\left|\bar{X}_s - \bar{X}_{k\Delta} \right| +\omega_{F,\sigma}(s-k\Delta) \right) ds  \right|^p\right] \\
&\lesssim
\sum_{k=0}^{[T/\Delta]-1} \int_{k\Delta}^{(k+1)\Delta} \E \left[  \left|\bar{X}_{s\wedge \tau^\varepsilon} - \bar{X}_{k\Delta\wedge \tau^\varepsilon} \right|^p  +\omega_{F,\sigma}(\Delta)^p \right] ds\\
&\lesssim \Delta^{p/2} +\omega_{F,\sigma}(\Delta)^p .
\end{align*}

For the last sum $\sum J^k_5$, by Burkholder-Davis-Gundy's inequality and \autoref{lem:discrXbis},
\begin{align*}
%
%%%%%%%%%%%%%%%%%%%%% J_5
%
\E \left[ \sup_{\substack{h=1,\dots,[T/\Delta]\\h\Delta \leq \tau^\varepsilon}} \left| 
\sum_{k=0}^{h-1} J^k_5 \right|^p\right] &\lesssim
\E \left[ \sup_{\substack{h=1,\dots,[T/\Delta]\\h\Delta \leq \tau^\varepsilon}} \left| 
\sum_{k=0}^{h-1} \int_{k\Delta}^{(k+1)\Delta} \left( \sigma(s,\bar{X}_s) - \sigma(k\Delta,\bar{X}_{k\Delta}) \right) dW_s  \right|^p\right] \\
&\lesssim
\E \left[ \left| 
\sum_{k=0}^{[T/\Delta]-1} \int_{k\Delta \wedge\tau^\varepsilon}^{(k+1)\Delta\wedge\tau^\varepsilon} 
\left| \sigma(s,\bar{X}_s) - \sigma(k\Delta,\bar{X}_{k\Delta}) \right|^2 ds  \right|^{p/2}\right] \\
&\lesssim
\E \left[ \left| 
\sum_{k=0}^{[T/\Delta]-1} \int_{k\Delta \wedge\tau^\varepsilon}^{(k+1)\Delta\wedge\tau^\varepsilon} 
\left| \sigma(s,\bar{X}_s) - \sigma(k\Delta,\bar{X}_{k\Delta}) \right|^2 ds  \right|^{p}\,\right]^{1/2} \\
&\lesssim
\left( \sum_{k=0}^{[T/\Delta]-1} \int_{k\Delta}^{(k+1)\Delta}
\E \left[  \left|\bar{X}_{s\wedge \tau^\varepsilon} - \bar{X}_{k\Delta\wedge \tau^\varepsilon} \right|^{2p} +(s-k\Delta)^{2p} \right] ds \right)^{1/2} \lesssim \;\Delta^{p/2}. 
\end{align*}
\end{proof}
\begin{rmk}\rm\label{choiceDelta}
The estimates given in \autoref{lem:aux} motivate the following choice of how 
$\Delta=\Delta_\varepsilon$ should behave when $\varepsilon$ goes to zero:
\[
\Delta^2/\varepsilon^3 \to 0,\quad
\omega_\sigma(\Delta)\Delta/\varepsilon^2 \to 0,\quad
\varepsilon^2/\Delta \to 0.
\]
Such a choice is always possible. Indeed, without loss of generality, 
we can suppose $\omega_\sigma(t)>t^{1/2}$, for every $t \in [0,T]$, 
and then define $\Delta=\Delta_\varepsilon$ via
$\Delta \sqrt{\omega_\sigma(\Delta)} = \varepsilon^2$.
%\begin{align*}
%\frac{\Delta^2}{\varepsilon^3} = \frac{\Delta^2}{\Delta^{3/2} \omega_\sigma(\Delta)^{3/4}} < \Delta^{1/8} \to 0, \quad
%\frac{\omega_\sigma(\Delta)\Delta}{\varepsilon^2} = \frac{\varepsilon^2}{\Delta} =
%\omega_\sigma(\Delta)^{1/2} \to 0.
%\end{align*}
\end{rmk}

%Finally, if $T/\Delta$ is not an integer, one can replace $T$ with $T+\delta$, for a little $\delta < \Delta$: therefore we will suppose in the following that $T/\Delta$ is an integer. 
We now discuss the 2nd sum on the right-hand side of \eqref{contribute}, that is
\begin{equation*}
\sum_{k=0}^{h-1} 
\left(
\int_{k\Delta}^{(k+1)\Delta} \left( \int_{k\Delta}^s  D\sigma(k\Delta,\bar{X}_{k\Delta})\sigma(k\Delta,\bar{X}_{k\Delta})  dW^\varepsilon_r \right) dW^\varepsilon_s - \int_{k\Delta}^{(k+1)\Delta} C(k\Delta,\bar{X}_{k\Delta}) ds
\right),
\end{equation*}
the $i$-th component of which, when plugging in \eqref{eq:Strat}, reads 
\begin{equation*}
\sum_{k=0}^{h-1} \sum_{\ell,m \in \N} \sum_{j=1,\dots,d}D_j \sigma^{i,m}(k\Delta,\bar{X}_{k\Delta})\sigma^{j,\ell}(k\Delta,\bar{X}_{k\Delta}) \left( c_{\ell,m}^k(\Delta,\varepsilon) -\delta_{\ell,m}\frac{q_m}{2} \Delta \right),
\end{equation*}
where $c_{\ell,m}^k(\Delta,\varepsilon)$ is given by
\begin{equation*}
c_{\ell,m}^k(\Delta,\varepsilon) 
= \int_{k\Delta}^{(k+1)\Delta} \left( \int_{k\Delta}^s dW^{\varepsilon,\ell}_r \right) dW^{\varepsilon,m}_s.
\end{equation*}

Taking the conditional expectation of
$c_{\ell,m}^k(\Delta,\varepsilon)$ with respect to $\mathcal{F}_{k\Delta}$ yields
\begin{align*}
\E \left[ c_{\ell,m}^k(\Delta,\varepsilon) \mid \mathcal{F}_{k\Delta}\right] 
= &\int_{k\Delta}^{(k+1)\Delta} \left( \int_{k\Delta}^s \E \left[ Y^{\varepsilon,\ell}_r Y^{\varepsilon,m}_s \mid \mathcal{F}_{k\Delta}\right] dr \right) ds
\\
= &Y^{\varepsilon,\ell}_{k\Delta}
Y^{\varepsilon,m}_{k\Delta} 
\int_{k\Delta}^{(k+1)\Delta} \left( \int_{k\Delta}^s   e^{-\varepsilon^{-2}(r+s-2k\Delta)} dr \right) ds 
\\
&+ \delta_{\ell,m}
\int_{k\Delta}^{(k+1)\Delta} \left( \int_{k\Delta}^s
q_\ell \frac{\varepsilon^{-2}}{2} 
\left( e^{-\varepsilon^{-2}(s-r)} - e^{-\varepsilon^{-2}(r+s-2k\Delta)} \right) 
dr \right) ds,
\end{align*}
where the following representation of $Y^\varepsilon$,
\begin{align*}
Y^{\varepsilon,m}_{s}  = 
Y^{\varepsilon,m}_{k\Delta} 
e^{-\varepsilon^{-2}(s-k\Delta)} + 
\int_{k\Delta}^{s} e^{-\varepsilon^{-2}(s-r)} \varepsilon^{-2} dW^m_r,
\end{align*}
has been used, and this conditional expectation can easily be calculated as
\begin{align} \label{eq:condE}
\E \left[ c_{\ell,m}^k(\Delta,\varepsilon) \mid \mathcal{F}_{k\Delta}\right] 
&= 
\frac{\varepsilon^4}{2} 
Y^{\varepsilon,\ell}_{k\Delta}
Y^{\varepsilon,m}_{k\Delta} 
\left( e^{-\varepsilon^{-2}\Delta} -1 \right)^2
+
\delta_{\ell,m}\frac{q_m}{2} \left( \Delta + \varepsilon^2 \left( - \frac{3}{2} + 2 e^{-\varepsilon^{-2}\Delta} - \frac{1}{2} e^{-2\varepsilon^{-2}\Delta}\right)  \right).
\end{align}

Now, since 
$\sum_{j=1,\dots,d} D_j \sigma^{i,m}(k\Delta,\bar{X}_{\tau^\varepsilon\wedge(k\Delta)})
\sigma^{j,\ell}(k\Delta,\bar{X}_{\tau^\varepsilon\wedge(k\Delta)})$ 
is $\mathcal{F}_{k\Delta}$ measurable, for every $\ell,m \in \N$, $i=1,\dots,d$, 
each process $M^i_h,\,h=1,\dots,[T/\Delta]$, given by
\begin{equation*}% \label{eq:nakaoM}
M^i_h =  \sum_{k=0}^{h-1} \sum_{\ell,m \in \N} \sum_{j=1,\dots,d}
D_j \sigma^{i,m}(k\Delta,\bar{X}_{\tau^\varepsilon\wedge(k\Delta)})
\sigma^{j,\ell}(k\Delta,\bar{X}_{\tau^\varepsilon\wedge(k\Delta)}) \left( c_{\ell,m}^k(\Delta,\varepsilon) -\E \left[ c_{\ell,m}^k(\Delta,\varepsilon) \mid \mathcal{F}_{k\Delta}\right]  \right),
\end{equation*}
is a discrete martingale with respect to the filtration 
$(\mathcal{F}_{h\Delta})_{h=1}^{[T/\Delta]}$. 
\begin{lem} \label{lem:M}
For each $i=1,\dots,d$,
\begin{align*}
\E \left[ \sup_{\substack{h=1,\dots,[T/\Delta]\\h\Delta \leq \tau^\varepsilon}} \left| 
M^i_h \right|^2\right]
\lesssim \left( \frac{\Delta}{\varepsilon}\right)^2 + \Delta .
\end{align*}
\end{lem}
\begin{proof}
Combining Doob's maximal inequality and martingale property gives
\begin{align*}
\E \left[ \sup_{\substack{h=1,\dots,[T/\Delta]\\h\Delta \leq \tau^\varepsilon}} \left| 
M^i_h \right|^2\right] 
&\lesssim
\E \left[ \left| M^i_{[T/\Delta]} \right|^2\right] \\
&\lesssim 
\sum_{k=0}^{[T /\Delta]-1} \E \left[ \left| \sum_{\ell,m \in \N} 
c_{\ell,m}^{k}(\Delta,\varepsilon) 
-\E \left[ c_{\ell,m}^{}(\Delta,\varepsilon) \mid \mathcal{F}_{k\Delta}\right] \right|^2 \right],
\end{align*}
where
\begin{align*}
\E \left[ \left| \sum_{\ell,m \in \N} c_{\ell,m}^{k}(\Delta,\varepsilon) 
-\E \left[ c_{\ell,m}^{k}(\Delta,\varepsilon) \mid \mathcal{F}_{k\Delta}\right] \right|^2 \right]
&\lesssim
\E \left[ \left| \sum_{\ell,m \in \N} c_{\ell,m}^{k}(\Delta,\varepsilon) \right|^2 \right],
\end{align*}
for each $k=0,\dots,[T/\Delta]-1$,
because the conditional expectation is an $L^2$-projection.
Thus, by independence of $Y^{\varepsilon,\ell}$ and $Y^{\varepsilon,m}$, for every $\ell \neq m$, we can estimate
\begin{align*}
\E \left[ \sup_{\substack{h=1,\dots,T/\Delta\\h\Delta \leq \tau^\varepsilon}} \left| 
M^i_h \right|^2\right] 
&\lesssim
\sum_{k=0}^{T/\Delta-1} \sum_{\ell,m \in \N} 
\E \left[ \left| \int_{k\Delta}^{(k+1)\Delta} \left( W^{\varepsilon,\ell}_s -W^{\varepsilon,\ell}_{k\Delta} \right) dW^{\varepsilon,m}_s \right|^2 \right] \\
&\lesssim
\sum_{k=0}^{T/\Delta-1}  \sum_{\ell,m \in \N}
\Delta \int_{k\Delta}^{(k+1)\Delta} \E \left[ \left| \left( W^{\varepsilon,\ell}_s -W^{\varepsilon,\ell}_{k\Delta} \right) Y^{\varepsilon,m}_s \right|^2 \right] ds\\
&\lesssim
\sum_{k=0}^{T/\Delta-1}  \sum_{\ell,m \in \N}
\Delta \int_{k\Delta}^{(k+1)\Delta} \E \left[  \left| W^{\varepsilon,\ell}_s - W^{\varepsilon,\ell}_{k\Delta} \right|^{2q} \right]^{1/q} \E \left[ \left| Y^{\varepsilon,m}_s \right|^{2q'} \right]^{1/q'} ds\\
&\lesssim
 \sum_{k=0}^{T/\Delta-1} 
\sum_{\ell,m \in \N} q_\ell q_m  \Delta \varepsilon^{-2} \int_{k\Delta}^{(k+1)\Delta} \left( \Delta + \varepsilon^2 \right) ds \lesssim
\left( \frac{\Delta}{\varepsilon}\right)^2 + \Delta.
\end{align*}
\end{proof}

To eventually cover the remainder of the 2nd sum on the right-hand side of \eqref{contribute},
after subtracting the martingale term $M_h$,
we introduce
\begin{equation*}% \label{eq:nakaoN}
N^i_h = \sum_{k=0}^{h-1} \sum_{\ell,m \in \N}\sum_{j=1,\dots,d}
D_j \sigma^{i,m}(k\Delta,\bar{X}_{k\Delta})\sigma^{j,\ell}(k\Delta,\bar{X}_{k\Delta})
\left(
\E \left[ c_{\ell,m}^k(\Delta,\varepsilon) \mid \mathcal{F}_{k\Delta}\right]
 - \delta_{\ell,m}\frac{q_m}{2} \Delta \right).
\end{equation*}
\begin{lem} \label{lem:N}
For each $i=1,\dots,d$,
%Let $N^i_h$ be defined by \eqref{eq:nakaoN} above. Then
\begin{equation*}
\E \left[ \sup_{\substack{h=1,\dots,[T/\Delta]\\h\Delta \leq \tau^\varepsilon}} \left| 
N^i_h \right|^2\right] \lesssim  \left( \frac{\varepsilon^2}{\Delta}\right)^2.
\end{equation*}
\end{lem}
\begin{proof}
The proof is an easy consequence of \eqref{eq:condE}. Indeed,
\begin{align*}
\E \left[ \sup_{\substack{h=1,\dots,[T/\Delta]\\h\Delta \leq \tau^\varepsilon}} \left| 
N^i_h \right|^2\right] &\lesssim 
\E \left[ \sup_{\substack{h=1,\dots,[T/\Delta]\\h\Delta \leq \tau^\varepsilon}} \left| 
\sum_{k=0}^{h-1} \sum_{\ell,m \in \N} \left|\E \left[ c_{\ell,m}^k(\Delta,\varepsilon) \mid \mathcal{F}_{k\Delta}\right] - \delta_{\ell,m}\frac{q_m}{2} \Delta \right| \right|^2\right] \\
&\lesssim 
\varepsilon^4 \Delta^{-1} \sum_{k=0}^{[T/\Delta]-1}
\sum_{\ell,m \in \N} q_\ell q_m
\lesssim 
\left( \frac{\varepsilon^2}{\Delta}\right)^2.
\end{align*}
\end{proof}

All in all, \autoref{lem:M} and \autoref{lem:N} together imply
\begin{align*}
\E \left[ \sup_{\substack{h=1,\dots,[T/\Delta]\\h\Delta \leq \tau^\varepsilon}} \left| 
\sum_{k=0}^{h-1} 
\left( I^k_6 - J^k_4 \right)  \right|^2\right] 
&= 
\E \left[ \sup_{\substack{h=1,\dots,[T/\Delta]\\h\Delta \leq \tau^\varepsilon}} \left| 
\left( M_h + N_h \right) \right|^2\right] 
\lesssim 
\left( \frac{\Delta}{\varepsilon}\right)^2 + \Delta +\left( \frac{\varepsilon^2}{\Delta}\right)^2,
\end{align*}
showing that the 2nd sum on the right-hand side of \eqref{contribute}
can be neglected, like the 5th one, when $\varepsilon\downarrow 0$,
and $\Delta=\Delta_\varepsilon$ behaves as described in \autoref{choiceDelta}.

Recall that we wanted  to control the remaining sums in terms of the difference $X^\varepsilon - \bar{X}$ itself,
which is obvious for the first and third sum on the right-hand side of \eqref{contribute}.
However, in case of the fourth sum, 
applying almost the same martingale argument used in case of the 2nd sum,
each term $I^k_5$ can be formally replaced by 
$\int_{k\Delta}^{(k+1)\Delta} \left( C(k\Delta,X^\varepsilon_{k\Delta}) - C(k\Delta,\bar{X}_{k\Delta}) \right) ds $,
subject to a sufficiently small $\varepsilon$-correction,
eventually leading to the wanted contraction argument in this case, too.

On the whole, we have justified that,
if $\Delta=\Delta_\varepsilon$ behaves as described in \autoref{choiceDelta},
then
\[
\E \left[ \sup_{\substack{k'=0,\dots,h\\k'\!\Delta \leq \tau^\varepsilon}}
\left| X^\varepsilon_{k'\Delta} - \bar{X}_{k'\Delta} \right|^2 \right] 
\lesssim \;r(\Delta,\varepsilon) +
\sum_{k=0}^{h-1} 
\Delta \E \left[ \sup_{\substack{k'=0,\dots,k\\k'\!\Delta \leq \tau^\varepsilon}}
\left| X^\varepsilon_{k'\Delta} - \bar{X}_{k'\Delta} \right|^2 \right],
\quad h=1,\dots,[T/\Delta],
\]
where $r(\Delta,\varepsilon) \to 0,\,\varepsilon \downarrow 0$,
finally proving \eqref{eq:Esup2}, by Gronwall's lemma.

The proof of \autoref{thm:main}(ii) is thus complete.
\section{Weak convergence} \label{sec:weakWZ}
In this section we prove part (i) of \autoref{thm:main}.
The idea of proof is similar to the one of part (ii), 
except that now $\beta\not=0$ is possible.
It is the existence of this bilinear term which prevents us from proving convergence in probability---we only
succeed in showing convergence in law (see \autoref{rmk:conv}(ii)).

First, we prove weak convergence of the bilinear term.

Second, we prove convergence in law of $X^\varepsilon,\,\varepsilon\downarrow 0$,
using bounds similar to those obtained in \autoref{sec:main}.
\subsection{Weak convergence of the bilinear term}  
For any $\varepsilon>0$, define the process $U^\varepsilon$ by
\begin{equation} \label{eq:Xepstild}
U^\varepsilon_t = \int_0^t \varepsilon \beta(Y^\varepsilon_s,Y^\varepsilon_s) ds,
\quad t\in[0,T],
\end{equation}
where $Y^\varepsilon$ is the stationary Ornstein-Uhlenbeck process introduced in \autoref{realTimeWipro}.
By (A5), the process $U^\varepsilon$ has zero-mean, 
and, using (A3), its second moments,
\[
\E \left[  
\int_0^t \varepsilon 
\underbrace{
\langle\beta(Y^\varepsilon_s,Y^\varepsilon_s) , \mathbf{e}_i \rangle
}_{\beta^i(Y^\varepsilon_s,Y^\varepsilon_s)} ds 
\int_0^t \varepsilon 
\underbrace{
\langle\beta(Y^\varepsilon_r,Y^\varepsilon_r) , \mathbf{e}_j \rangle 
}_{\beta^j(Y^\varepsilon_r,Y^\varepsilon_r)} dr\right],
\]
can be calculated to be
\[
\frac{1}{2} \sum_{\ell, m \in \N} 
\underbrace{
\langle\beta(\mathbf{f}_\ell,\mathbf{f}_m) , \mathbf{e}_i \rangle
}_{\beta^i_{\ell,m} }
\underbrace{
\langle\beta(\mathbf{f}_\ell,\mathbf{f}_m) , \mathbf{e}_j \rangle
}_{\beta^j_{\ell,m} }
q_\ell q_m  \left( t + \frac{\varepsilon^2}{2}\left( e^{-2\varepsilon^{-2}t}-1\right) \right),
\]
for $i,j=1,\dots,d$, and $\ell,m \in \N$.

Recalling \eqref{diffCoeff}, using the above short notation, we also have that
\begin{align*}
b^i_{\ell,m}
= \beta^i_{\ell,m} \sqrt{\frac{q_\ell q_m}{2}},
\quad  i=1,\dots,d,\;\ell,m\in\N.
\end{align*}

Next, since 
$dY^{\varepsilon,\ell}_t = -\varepsilon^{-2}Y^{\varepsilon,\ell}_t dt 
+ \varepsilon^{-2} d\langle W_t , \mathbf{f}_\ell \rangle$, 
It\^o's formula implies
\begin{align*}
Y^{\varepsilon,\ell}_t Y^{\varepsilon,m}_t
=
Y^{\varepsilon,\ell}_0 Y^{\varepsilon,m}_0 
- 2\varepsilon^{-2}\hspace{-5pt} \int_0^t 
Y^{\varepsilon,\ell}_s Y^{\varepsilon,m}_s ds
+\varepsilon^{-2}\hspace{-5pt} \int_0^t
Y^{\varepsilon,\ell}_s d\langle W_s , \mathbf{f}_m \rangle 
+\varepsilon^{-2}\hspace{-5pt} \int_0^t
Y^{\varepsilon,m}_s d\langle W_s , \mathbf{f}_\ell \rangle
+ \frac{t \varepsilon^{-4}}{2} q_\ell \delta_{\ell,m},
\end{align*}
for any $\ell,m\in\N$, and hence
\begin{align*}
U^{\varepsilon,i}_t
=
\int_0^t \varepsilon \sum_{\ell,m \in \N} \beta^i_{\ell,m} Y^{\varepsilon,\ell}_s Y^{\varepsilon,m}_s  ds 
=
&\;\varepsilon
\int_0^t \sum_{\ell,m \in \N} \beta^i_{\ell,m}
Y^{\varepsilon,\ell}_s d\langle W_s , \mathbf{f}_m \rangle\\
&- \frac{\varepsilon^3}{2}
\sum_{\ell,m \in \N} \beta^i_{\ell,m} \left( 
Y^{\varepsilon,\ell}_t Y^{\varepsilon,m}_t
- Y^{\varepsilon,\ell}_0 Y^{\varepsilon,m}_0\right)
\,+\,\frac{\varepsilon^{-1}}{4}\,t\sum_{\ell\in\N}\beta^i_{\ell,\ell}q_\ell
\\
=
&\;M^{\varepsilon,i}_t - \frac{1}{2} V^{\varepsilon,i}_t
\,+\,\frac{\varepsilon^{-1}}{4}\,t\sum_{\ell\in\N}\beta^i_{\ell,\ell}q_\ell,
\end{align*}
where $M^\varepsilon$ is a $d$-dimensional continuous local martingale, 
while the process $V^\varepsilon$ satisfies
\begin{align} \label{eq:V}
\E \left[ \sup_{t \leq T} \left| V^{\varepsilon}_t \right|^p \right]
=
\E \left[ \sup_{t \leq T} \left| \varepsilon^3 
\left( \beta \left( 
Y^{\varepsilon}_t,Y^{\varepsilon}_t\right)
-
\beta \left( 
Y^{\varepsilon}_0,Y^{\varepsilon}_0\right) \right)
 \right|^p \right]
\lesssim
\varepsilon^p,\quad\forall\,p>1,
\end{align}
by combining (A3) and \autoref{lem:ou}.

The above representation of $U^\varepsilon$,
though very simple, has been used in a variety of cases in a fruitful way,
see for instance \cite{Ol94} or \cite{IfPaPi08}.
Observe that, by (A5), the It\^o-correction actually cancels out, being otherwise a contribution of order $\varepsilon^{-1}$.
The process $U^\varepsilon$, nevertheless, has got an interesting limit in law:
\begin{prop} \label{prop:conv}
The couple of processes $(U^\varepsilon,W)$ converges in law, $\varepsilon\downarrow 0$,
to a pair of processes $(\eta,\omega)$, where
$\eta$ is a $d$-dimensional Wiener process with covariance
$(\sum_{\ell,m \in \N} b^i_{\ell,m}b^j_{\ell,m})_{i,j=1}^d$,
and $\omega$ is a $Q$-Wiener process, like $W$. %as described in \autoref{realTimeWipro}.
Furthermore, $\eta$ and $\omega$ are independent.
\end{prop}
\begin{proof}
First, by \eqref{eq:V}, it is sufficient to prove the proposition for $(M^\varepsilon,W)$
instead of $(U^\varepsilon,W)$. 

Since all components of the processes $M^\varepsilon,\,\varepsilon>0$, and of $W$,
are continuous local martingales, the distributional properties of the limit $(\eta,\omega)$
would follow from \cite[Chapter VII, Theorem 1.4]{EtKu86}, if
\begin{align*}
\E \left[ 
\left( 
\left[ M^{\varepsilon,i} , M^{\varepsilon,j} \right]_{t} - t \sum_{\ell,m \in \N} 
b^i_{\ell,m}b^j_{\ell,m}
\right)^2
\right] \to 0,\quad\varepsilon\downarrow 0, \numberthis\label{toShow}
\end{align*}
for each $t \in [0,T]$, and $i,j =1,\dots, d$, as well as
\[
\E\left[
(\left[ M^{\varepsilon,i} , \langle W , \mathbf{f}_m \rangle \right]_{t}\,)^2
\right]\to 0,\quad\varepsilon\downarrow 0, 
\]
for each $t\in[0,T],\,i=1,\dots,d$, and $m\in\N$.

First, fix $t \in [0,T]$, as well as $i,j =1,\dots, d$.
Then,
the quadratic covariation $\left[ M^{\varepsilon,i} , M^{\varepsilon,j} \right]_{t}$ is given by
\begin{align*}
\left[ M^{\varepsilon,i} , M^{\varepsilon,j} \right]_{t} 
=
\varepsilon^2 
\int_0^t 
\sum_{m \in \N}
\sum_{\ell,\ell' \in \N}
\beta^i_{\ell,m} \beta^j_{\ell',m} q_m  
Y^{\varepsilon,\ell}_s Y^{\varepsilon,\ell'}_s ds,
\end{align*}
so that
\begin{align*}
\E& \left[ 
\left( 
\left[ M^{\varepsilon,i} , M^{\varepsilon,j} \right]_{t} - t \sum_{\ell,m \in \N} 
b^i_{\ell,m}b^j_{\ell,m}
\right)^2
\right]
\\ &\quad= 
\varepsilon^4
\iint_0^t 
\sum_{m,\underline{m}\in \N}
\sum_{\substack{\ell,\ell' \in \N\\ \underline{\ell} ,\underline{\ell}' \in \N}}
\beta^i_{\ell,m} 
\beta^j_{\ell',m} 
\beta^i_{\underline{\ell},\underline{m}} \beta^j_{\underline{\ell}',\underline{m}}
q_m q_{\underline{m}} \E \left[
Y^{\varepsilon,\ell}_s 
Y^{\varepsilon,\ell'}_s 
Y^{\varepsilon,\underline{\ell}}_r 
Y^{\varepsilon,\underline{\ell}'}_r \right]
ds dr
\\
&\quad\quad-2
\varepsilon^2 
\int_0^t 
\sum_{m \in \N}
\sum_{\ell,\ell' \in \N}
\beta^i_{\ell,m} \beta^j_{\ell',m} q_m  \E \left[
Y^{\varepsilon,\ell}_s Y^{\varepsilon,\ell'}_s \right] ds
\left( t \sum_{\ell,m \in \N} b^i_{\ell,m} b^j_{\ell,m}\right)
+
\left( t \sum_{\ell,m \in \N} b^i_{\ell,m} b^j_{\ell,m}\right)^2.
\end{align*}

Now, using that one can easily calculate
$\E \left[
Y^{\varepsilon,\ell}_s Y^{\varepsilon,\ell'}_s \right] 
= 
\frac{\varepsilon^{-2}}{2} q_\ell \delta_{\ell,\ell'}$,
it follows from Isserlis-Wick's theorem, see \cite[Theorem 1.28]{Ja97}, that
\begin{align*}
\E \left[
Y^{\varepsilon,\ell}_s 
Y^{\varepsilon,\ell'}_s 
Y^{\varepsilon,\underline{\ell}}_r 
Y^{\varepsilon,\underline{\ell}'}_r \right]
&= \frac{\varepsilon^{-4}}{4}
\left( 
q_\ell q_{\underline{\ell}} 
\delta_{\ell,\ell'}
\delta_{\underline{\ell},\underline{\ell}'}
+ q_\ell q_{\ell'} e^{-2\varepsilon^{-2}|s-r|}
\left( 
\delta_{\ell,\underline{\ell}}
\delta_{\ell',\underline{\ell}'} 
+\delta_{\ell,\underline{\ell}'}
\delta_{\ell',\underline{\ell}}
\right) \right),
\end{align*}
which yields
\begin{align*}
\E \left[ 
\left( 
\left[ M^{\varepsilon,i} , M^{\varepsilon,j} \right]_{t} - t \sum_{\ell,m \in \N}
b^i_{\ell,m}b^j_{\ell,m}
\right)^2
\right]
=
\left( t \sum_{\ell,m \in \N} b^i_{\ell,m} b^j_{\ell,m} 
-t \sum_{\ell,m \in \N} b^i_{\ell,m} b^j_{\ell,m}\right)^2 + O(\varepsilon^2)
\lesssim
\varepsilon^2,
\end{align*}
proving \eqref{toShow}.

Second, fix $t\in[0,T]$, as well as $i=1,\dots,d,\,m\in\N$. Then,
\[
\left[M^{\varepsilon,i} , \langle W , \mathbf{f}_m \rangle\right]_t
\,=\,
\int_0^{t}\beta^i(\varepsilon Y^\varepsilon_s,Q\mathbf{f}_m)\,ds,
\]
where, using \autoref{lem:ou},
\[
\E\left[%\sup_{t\le T}
|\!\int_0^t\!\beta^i(\varepsilon Y^\varepsilon_s,Q\mathbf{f}_m)\,ds\,|^2
\right]
=\;
\E\left[%\sup_{t\le T}
|\,\beta^i(\varepsilon\!\int_0^t\!Y^\varepsilon_s ds\,,Q\mathbf{f}_m)\,|^2
\right]
\lesssim\;
\E\left[\rule{0pt}{12pt}\right.%\sup_{t\le T}
|\hspace{-8pt}
\overbrace{\varepsilon\!\int_0^t\!Y^\varepsilon_s ds}^{
\varepsilon W_t-\varepsilon^3(Y^\varepsilon_t-Y^\varepsilon_0)}
\hspace{-8pt}|^2\,q_m^2
\left.\rule{0pt}{12pt}\right]
\stackrel{\varepsilon\downarrow 0}{\longrightarrow}
\;0,
\]
finishing the proof of the proposition.
\end{proof}
\begin{rmk} \label{rmk:conv}
(i) 
Of course, a $d$-dimensional Wiener process with covariance
$(\sum_{\ell,m \in \N} b^i_{\ell,m}b^j_{\ell,m})_{i,j=1}^d$
can always be represented  by
$\sum_{\ell,m \in \N} b_{\ell,m} \bar{W}^{\ell,m}$, 
where $\{\bar{W}^{\ell,m}\}_{\ell,m \in \N}$ is a family of independent 
one-dimensional standard Wiener processes.

(ii)
We would like to stress that we do not expect a much stronger convergence of $U^\varepsilon$,
when $\varepsilon\downarrow 0$,
as the one stated in the above proposition. Indeed, it turns out to be that
the sequence $\{M^{\varepsilon}\}_{\varepsilon>0}$ 
is not even a Cauchy sequence in $L^2(\Omega;\R^d)$. 
To see this, for fixed $0<\varepsilon<\underline{\varepsilon}$, and some $1\le i\le d$, consider
\begin{align*}
\mathbb{E} \left[
\sup_{t \leq T} \left|
M^{\varepsilon,i}_t - 
M^{\underline{\varepsilon},i}_t 
\right|^2 \right] 
&=
\mathbb{E} \left[
\sup_{t \leq T} \left|
\int_0^t \sum_{\ell,m \in \N} \beta^i_{\ell,m}
\left( 
\varepsilon Y^{\varepsilon,\ell}_s -
\underline{\varepsilon} Y^{\underline{\varepsilon},\ell}_s \right)
d\langle  W_s , \mathbf{f}_m \rangle
\right|^2 \right].
\end{align*}
But, by Burkholder-Davis-Gundy's inequality, the above expectation can be bound from below by
\begin{align*}
\mathbb{E} \left[
\int_0^T \sum_{m \in \N} \left( 
\sum_{\ell \in \N} \beta^i_{\ell,m} 
\left( 
\varepsilon Y^{\varepsilon,\ell}_s -
\underline{\varepsilon} Y^{\underline{\varepsilon},\ell}_s \right) \right)^2 q_m 
ds \right]
&=
T \sum_{\ell,m \in \N} (\beta^i_{\ell,m})^2 q_\ell q_m
\left( 1 - \frac{2\varepsilon^{-1}\underline{\varepsilon}^{-1}}{\varepsilon^{-2}+\underline{\varepsilon}^{-2}} \right),
\end{align*}
where
\begin{align*}
\lim_{\varepsilon \to 0} \left( 1 - \frac{2\varepsilon^{-1}\underline{\varepsilon}^{-1}}{\varepsilon^{-2}+\underline{\varepsilon}^{-2}} \right) = 1,
\quad\mbox{for every fixed $\underline{\varepsilon}>0$},
\end{align*}
so that $\{M^{\varepsilon,i}\}_{\varepsilon>0}$ cannot be Cauchy in $L^2(\Omega)$. 
\end{rmk}
\subsection{Weak convergence of solutions}
We now prove $X^\varepsilon\to\bar{X}$, in law, when $\varepsilon\downarrow 0$.

First, for each $\varepsilon>0$, let $\hat{X}^\varepsilon$ be the solution  of 
\begin{align*}
\hat{X}^\varepsilon_t = x_0 + \int_0^t \left( F(s,\hat{X}^\varepsilon_s) + C(s,\hat{X}^\varepsilon_s)\right) ds + \int_0^t \sigma(s,\hat{X}^\varepsilon_s ) dW_s + U^\varepsilon_t,
\quad t\in[0,T], \numberthis\label{auxEqu}
\end{align*}
where $U^\varepsilon$ is given by \eqref{eq:Xepstild},
and let 
$\tau^\varepsilon_R = \inf \{t\geq 0 : |X^\varepsilon_t|\geq R \} \wedge \inf \{t\geq 0 : |\hat{X}^\varepsilon_t|\geq R \} $.

Note that, if (A4),
then the coefficients $F,C,\sigma,\beta$ must have properties 
such that each of the above equations admits global solutions on $[0,T]$, too.

Next, taking into account
$\E \left[ \left| \varepsilon \beta(Y^\varepsilon_s,Y^\varepsilon_s) \right|^p \right] \lesssim \varepsilon^{-p}$
as well as
\begin{align*} 
\E \left[ |{M}^\varepsilon_{(k+1)\Delta\wedge\tau^\varepsilon_R} - {M}^\varepsilon_{k\Delta\wedge\tau^\varepsilon_R}|^p \right] 
\lesssim
\E \left[\left| \int_{k\Delta\wedge\tau^\varepsilon_R}^{(k+1)\Delta\wedge\tau^\varepsilon_R}
\sum_{m \in \N} \left( \sum_{\ell \in \N}
\beta_{\ell,m} \varepsilon Y^{\varepsilon,\ell}_s \right)^2 q_m ds \right|^{p/2}\right]
\lesssim
\Delta^{p/2},
\end{align*}
it can easily be verified that 
\autoref{lem:discrXeps} \&\ \autoref{lem:discrXepsbis} would still be valid, despite $\beta\not=0$,
on the one hand,
and that the following versions 
\[
\E \left[ \sup_{t \leq \tau,\,t+k\Delta \leq T\wedge\tau^\varepsilon_R}|\hat{X}^\varepsilon_{t+k\Delta} - \hat{X}^\varepsilon_{k\Delta}|^p \right] \lesssim \tau^{\frac{p}{2}} + \left( \frac{\tau}{\varepsilon}\right)^p, 
\quad p>1,\,\tau\in(0,1),\,k\in\{0,1,\dots,[T/\Delta]\},
\]
and
\[
\E \left[ \sup_{\substack{k=0,1,\dots,[T/\Delta] \\t \leq \Delta ,\, t+k\Delta \leq T\wedge\tau^\varepsilon_R }} |\hat{X}^\varepsilon_{t+k\Delta} - \hat{X}^\varepsilon_{k\Delta}|^p \right] \lesssim \Delta^{\frac{p}{2}-1} 
+  \frac{\Delta^{p-1}}{\varepsilon^p},
\quad p>1,
\]
of  \autoref{lem:discrXbis} \&\ \autoref{lem:discrX}, respectively,
would hold true when replacing $\bar{X}$ by $\hat{X}^\varepsilon$, on the other.

Therefore, when expanding
$X^\varepsilon$ and $\hat{X}^\varepsilon$ as in \eqref{eq:incXeps} \&\ \eqref{eq:incX},
but including the $\beta$-term,
and then arguing as in the proof of \autoref{thm:main}(ii) in \autoref{sec:main}, 
it would immediately follow  that
$X^\varepsilon_{\cdot \wedge \tau^\varepsilon_R} - \hat{X}^\varepsilon_{\cdot \wedge \tau^\varepsilon_R} \to 0$, 
in probability, $\varepsilon \downarrow 0$, for any $R>0$,
once the following lemma is also available.
\begin{lem}
Assume that $\Delta=\Delta_\varepsilon$ behaves as described in \autoref{choiceDelta}.
Then,
\begin{align*}
\E \left[ 
\sup_{\substack{h=1,\dots,[T/\Delta] \\ h\Delta \leq \tau^\varepsilon_R}}
\left| \sum_{k=0}^{h-1}
\int_{k\Delta}^{(k+1)\Delta} \left( \int_{k\Delta}^s D\sigma(r,X^\varepsilon_r) \varepsilon \beta(Y^\varepsilon_r,Y^\varepsilon_r) dr \right) dW^\varepsilon_s
\right|^2\right]
\to 0,\quad\varepsilon\downarrow 0.
\end{align*}
\end{lem}
\begin{proof}
To start with, write
\begin{align*}
\int_{k\Delta}^{(k+1)\Delta} &\left( \int_{k\Delta}^s D\sigma(r,X^\varepsilon_r) \varepsilon \beta(Y^\varepsilon_r,Y^\varepsilon_r) dr \right) dW^\varepsilon_s \\
=\,
&\int_{k\Delta}^{(k+1)\Delta} \left( \int_{k\Delta}^s \left(
D\sigma(r,X^\varepsilon_r) \varepsilon \beta(Y^\varepsilon_r,Y^\varepsilon_r) - D\sigma({k\Delta},X^\varepsilon_{k\Delta}) \varepsilon \beta(Y^\varepsilon_r,Y^\varepsilon_r) \right) dr \right) dW^\varepsilon_s \\
&+ 
\int_{k\Delta}^{(k+1)\Delta} \left( \int_{k\Delta}^s  D\sigma({k\Delta},X^\varepsilon_{k\Delta}) \varepsilon \beta(Y^\varepsilon_r,Y^\varepsilon_r)
dr \right) dW^\varepsilon_s,
\end{align*}
which creates two summands, for any fixed $0\le k\le[T/\Delta]-1$.

We estimate the impact of each summand separately.

First, using $|D\sigma(r,X^\varepsilon_r)-D\sigma({k\Delta},X^\varepsilon_{k\Delta})| \lesssim |X^\varepsilon_r-X^\varepsilon_{k\Delta}| + \omega_\sigma(\Delta)$, we obtain that
\begin{align*}
&\E \left[ 
\sup_{\substack{h=1,\dots,[T/\Delta] \\ h\Delta \leq \tau^\varepsilon_R}}
\left| \sum_{k=0}^{h-1}
\int_{k\Delta}^{(k+1)\Delta} 
\left( \int_{k\Delta}^s 
\left(
D\sigma(r,X^\varepsilon_r) \varepsilon \beta(Y^\varepsilon_r,Y^\varepsilon_r) - D\sigma({k\Delta},X^\varepsilon_{k\Delta}) \varepsilon \beta(Y^\varepsilon_r,Y^\varepsilon_r) \right) dr \right) dW^\varepsilon_s 
\right|^2\right]
\\
&\lesssim
\varepsilon^{-4}
\E \left[ 
\sup_{\substack{h=1,\dots,[T/\Delta] \\ h\Delta \leq \tau^\varepsilon_R}}
\left| \sum_{k=0}^{h-1}
\int_{k\Delta}^{(k+1)\Delta} 
\left( \int_{k\Delta}^s 
\left(
|X^\varepsilon_r-X^\varepsilon_{k\Delta}| + \omega_\sigma(\Delta) \right) dr \right) 
ds 
\right|^2 \right]
\\
&\lesssim
\varepsilon^{-4}
\E \left[ 
 \sum_{k=0}^{\lceil T \wedge \tau^\varepsilon_R\rceil / \Delta-1}
\int_{k\Delta}^{(k+1)\Delta} \left|
\int_{k\Delta}^s 
\left(
|X^\varepsilon_r-X^\varepsilon_{k\Delta}| + \omega_\sigma(\Delta) \right) dr 
\right|^2 ds 
 \right]
\\
&\lesssim
\varepsilon^{-4} 
 \sum_{k=0}^{[T / \Delta]-1}
\int_{k\Delta}^{(k+1)\Delta}
(s-k\Delta)
\int_{k\Delta}^s 
\left(
\E \left[|X^\varepsilon_{r\wedge \tau^\varepsilon_R}-X^\varepsilon_{k\Delta\wedge \tau^\varepsilon_R}|^2\right] + \omega_\sigma(\Delta)^2 \right) dr  ds
\lesssim \left( \frac{\Delta^2}{\varepsilon^3}\right)^2 + \left( \frac{\Delta}{\varepsilon^2}\right)^2 \omega_\sigma(\Delta)^2.
\end{align*}

Second, we approach 
\begin{equation}\label{approached}
\sum_{k=0}^{h-1}
\int_{k\Delta}^{(k+1)\Delta} \left( \int_{k\Delta}^s  D\sigma({k\Delta},X^\varepsilon_{k\Delta}) \varepsilon \beta(Y^\varepsilon_r,Y^\varepsilon_r)
dr \right) dW^\varepsilon_s
\end{equation}
following the method used when
discussing the 2nd sum on the right-hand side of \eqref{contribute}
in the proof of \autoref{thm:main}(ii), but now for triple moments of $Y^\varepsilon$.

Indeed, define
\begin{align*}
c^k_{\ell,m,n}(\Delta,\varepsilon)
=
\int_{k\Delta}^{(k+1)\Delta} \left( \int_{k\Delta}^s   Y^{\varepsilon,\ell}_r Y^{\varepsilon,m}_r 
dr \right) Y^{\varepsilon,n}_s ds,
\end{align*}
and take the conditional expectation with respect to $\mathcal{F}_{k\Delta}$, that is
\begin{align*}
\E \left[ c^k_{\ell,m,n}(\Delta,\varepsilon)
\mid \mathcal{F}_{k\Delta}\right] 
=
\int_{k\Delta}^{(k+1)\Delta} \left( \int_{k\Delta}^s   \E \left[ Y^{\varepsilon,\ell}_r Y^{\varepsilon,m}_r 
Y^{\varepsilon,n}_s \mid \mathcal{F}_{k\Delta}\right] dr \right)  ds.
\end{align*}
Since
\begin{align*}
\E \left[ Y^{\varepsilon,\ell}_r Y^{\varepsilon,m}_r 
Y^{\varepsilon,n}_s \mid \mathcal{F}_{k\Delta}\right]
=\,
&Y^{\varepsilon,\ell}_{k\Delta} 
Y^{\varepsilon,m}_{k\Delta} 
Y^{\varepsilon,n}_{k\Delta}
e^{-\varepsilon^{-2}(s+2r-3k\Delta)} \\
&+
\left( 
Y^{\varepsilon,\ell}_{k\Delta} \delta_{m,n} q_n+
Y^{\varepsilon,m}_{k\Delta} \delta_{\ell,n} q_n+
Y^{\varepsilon,n}_{k\Delta} \delta_{\ell,m} q_\ell\right)
\frac{\varepsilon^{-2}}{2}
\left( e^{-\varepsilon^{-2}(s-k\Delta)} - e^{-\varepsilon^{-2}(s+2r-3k\Delta)}\right),
\end{align*}
we have that
\begin{align*}
\E \left[ c^k_{\ell,m,n}(\Delta,\varepsilon)
\mid \mathcal{F}_{k\Delta}\right] 
=\,
&Y^{\varepsilon,\ell}_{k\Delta} 
Y^{\varepsilon,m}_{k\Delta} 
Y^{\varepsilon,n}_{k\Delta}
\frac{\varepsilon^4}{2}
\left( 1 - e^{-\varepsilon^{-2} \Delta} - \frac{1}{3} + \frac{1}{3} e^{-3\varepsilon^{-2}\Delta} \right)
\\
&+
\left( 
Y^{\varepsilon,\ell}_{k\Delta} \delta_{m,n} q_n+
Y^{\varepsilon,m}_{k\Delta} \delta_{\ell,n} q_n+
Y^{\varepsilon,n}_{k\Delta} \delta_{\ell,m} q_\ell\right)
\\
& \quad \times
\frac{\varepsilon^{2}}{2}
\left( 
\frac{\Delta}{\varepsilon^2}e^{-\varepsilon^{-2}\Delta}
+ \frac{1}{2} - \frac{1}{2} e^{-\varepsilon^{-2}\Delta} 
+ \frac{1}{6} - \frac{1}{6} e^{-3\varepsilon^{-2}\Delta}
\right).
\end{align*}

Next, for each $i=1,\dots,d$, the process
$M^i_h,\,h=1,\dots,[T/\Delta]$, given by 
\begin{align*}
M^i_h
=
\sum_{k=0}^{h-1}
\sum_{\ell,m,n \in \N}
\sum_{j =1,\dots,d}
D_j \sigma^{i,n}(k\Delta,X^\varepsilon_{\tau^\varepsilon_R\wedge k\Delta})
\varepsilon \beta^j_{\ell,m} \left( c^k_{\ell,m,n}(\Delta,\varepsilon)-\E \left[ c^k_{\ell,m,n}(\Delta,\varepsilon)
\mid \mathcal{F}_{k\Delta}\right] \right),
\end{align*}
is a martingale with respect to the filtration $(\mathcal{F}_{h\Delta})_{h=1}^{[T/\Delta]}$, 
and arguing as in the proof of \autoref{lem:M} yields
\begin{align*}
\E \left[ 
\sup_{\substack{h=1,\dots,[T/\Delta] \\ h\Delta \leq \tau^\varepsilon_R}}
\left| M^i_h
\right|^2\right]
\lesssim
\frac{\Delta^3}{\varepsilon^4},
\quad i=1,\dots,d.
\end{align*}

So, it remains to prove that the remainder,
after  subtracting the martingale term $M_h$ from \eqref{approached},
also vanishes, when $\varepsilon\downarrow 0$.
For $i=1,\dots,d$, the $i$th coordinate of this remainder reads
\begin{align*}
N^i_h
=
\sum_{k=0}^{h-1}
\sum_{\ell,m,n \in \N}
\sum_{j =1,\dots,d}
D_j \sigma^{i,n}(k\Delta,X^\varepsilon_{k\Delta})
\varepsilon B^j_{\ell,m} 
\E \left[ c^k_{\ell,m,n}(\Delta,\varepsilon)
\mid \mathcal{F}_{k\Delta}\right],
\end{align*}
and we can easily calculate the below bound,
\begin{align*}
\E \left[  
\sup_{\substack{h=1,\dots,T/\Delta \\ h\Delta \leq \tau^\varepsilon_R}}
\left| N^i_h
\right|^2 \right]
\lesssim
\Delta^{-1}
\sum_{k=0}^{T / \Delta - 1} 
\E \left[  \left| \varepsilon
\E \left[ c^k_{\ell,m,n}(\Delta,\varepsilon)
\mid \mathcal{F}_{k\Delta}\right]
\right|^2 \right]
\lesssim \left( \frac{\varepsilon^2}{\Delta}\right)^2,
\end{align*}
finishing the proof of the lemma.
\end{proof}
\begin{cor}
For any $R>0$,
if $\Delta=\Delta_\varepsilon$ behaves as described in \autoref{choiceDelta},
\[
\E \left[ \sup_{t \leq T\wedge\tau^\varepsilon_R }|X^\varepsilon_t - \hat{X}^\varepsilon_t|^2\right]
\to 0,\quad\varepsilon\downarrow 0,
\]
and hence
$X^\varepsilon_{\cdot \wedge \tau^\varepsilon_R} - \hat{X}^\varepsilon_{\cdot \wedge \tau^\varepsilon_R} \to 0$, 
in probability, $\varepsilon \downarrow 0$, in particular.
\end{cor}

The above corollary suggests that it would be sufficient to show that
$\hat{X}_{\cdot \wedge \tau^\varepsilon_R}^\varepsilon\to\bar{X}_{\cdot \wedge \tau^\varepsilon_R}$, 
in law, when $\varepsilon\downarrow 0$,
subject to some procedure allowing to let $R$ go to infinity, afterwards.
So, we at first prove the weak convergence for fixed $R$,
and then discuss the limit-procedure for $R\to\infty$.

Modify the coefficients $F,\sigma$ outside the set $\{(t,x): |x|<R\}$ in such a way 
that the new coefficients $F_R,\,\sigma_R$, but also $D\sigma_R$,
are globally bounded, and that 
both functions $F_R(t,\cdot)$ and $D\sigma_R(t,\cdot)$ are globally Lipschitz, 
uniformly in $t \in [0,T]$. 

Of course, $\hat{X}^\varepsilon_{\cdot \wedge \tau^\varepsilon_R}$ coincides with 
$\hat{X}^{\varepsilon,R}_{\cdot \wedge \tau^\varepsilon_R}$, 
where $\hat{X}^{\varepsilon,R}$ denotes the solution to the equation obtained when replacing
the coefficients of  \eqref{auxEqu} by $F_R,\,\sigma_R$, 
and the Stratonovich correction $C_R$ associated with $\sigma_R$.
Also, let $\bar{X}^R$ denote the solution to the equation obtained when replacing
the coefficients of  \eqref{eq:X} by $F_R,\,\sigma_R,\,C_R$. 
\begin{prop}
Fix $R>0$. Then, $\hat{X}^{\varepsilon,R}$ converges to $\bar{X}^R$, 
in law, when $\varepsilon\downarrow 0$.
\end{prop}
\begin{proof}
Since
\begin{align*}
\hat{X}^{\varepsilon,R}_t - U^\varepsilon_t
= x_0 + \int_0^t \left( F_R(s,\hat{X}^{\varepsilon,R}_s) + C_R(s,\hat{X}^{\varepsilon,R}_s)\right) ds + 
\int_0^t \sigma_R(s,\hat{X}^{\varepsilon,R}_s ) dW_s,
\end{align*}
by boundedness of the coefficients on the above right-hand side, we obtain that
\[
\E\left[\sup_{t\le T}
|\hat{X}^{\varepsilon,R}_t - U^\varepsilon_t|
\right]
\,\lesssim\,
|x_0| + T
+ \E\left[\sup_{t\le T}|\int_0^t \sigma_R(s,\hat{X}^{\varepsilon,R}_s ) dW_s|
\right],
\]
where Burkholder-Davis-Gundy's inequality gives
$\E\left[\sup_{t\le T}|\int_0^t \sigma_R(s,\hat{X}^{\varepsilon,R}_s ) dW_s|\right]\,\lesssim\,T^{1/2}$.

Similarly, 
$\E\left[|(\hat{X}^{\varepsilon,R}_{t_2} - U^\varepsilon_{t_2}) - (\hat{X}^{\varepsilon,R}_{t_1} - U^\varepsilon_{t_1})|^p\right]\lesssim|t_2-t_1|^{p/2}$,
for any $|t_2-t_1|<1$, and any $p>1$.
Thus, by Kolmogorov-Chentsov's theorem, for every $\alpha\in(0,1)$,
one can find $\Delta\in(0,1)$ such that
\[
\PP\left\{\sup_{t_1,t_2\in[0,T],\,|t_2-t_1|\le\Delta}
%{\substack{ 0\le t_1\le t_2\le T \\ t_2-t_1\le\Delta }}
\frac{
|(\hat{X}^{\varepsilon,R}_{t_2} - U^\varepsilon_{t_2}) - (\hat{X}^{\varepsilon,R}_{t_1} - U^\varepsilon_{t_1})|
}
{|t_2-t_1|^\gamma}
\,\le\,const
\right\}
\,\ge\,1-\alpha,\quad\forall\,\varepsilon>0,
\]
where $const$ depends on $\gamma$, but not on $\varepsilon$,
and $\gamma\in(0,1/2)$ can be freely chosen.

We therefore have equi-boundedness and equi-continuity
of $\{\hat{X}^{\varepsilon,R} - U^\varepsilon\}_{\varepsilon>0}$ 
with arbitrarily high probability, and hence the family $\{\hat{X}^{\varepsilon,R} - U^\varepsilon\}_{\varepsilon>0}$
is tight with respect to the uniform topology in $C([0,T],\R^d)$,
first applying Arzel\`a-Ascoli, followed by Prokhorov's theorem.
Moreover, $\{U^\varepsilon\}_{\varepsilon > 0}$ is trivially tight by \autoref{prop:conv},
so that adding $\hat{X}^{\varepsilon,R} - U^\varepsilon$ and $U^\varepsilon$
would make $\{\hat{X}^{\varepsilon,R}\}_{\varepsilon > 0}$ tight, too.

All in all, the family of triples
$\{\left(\rule{0pt}{9pt}\right.
\hat{X}^{\varepsilon,R},U^\varepsilon,{W}
\left.\rule{0pt}{9pt}\right)\}_{\varepsilon > 0}$ is tight.

Next, for $\varepsilon>0$, 
let $\PP^{R,\varepsilon}$ be the pushforward measure
$\PP\circ(\hat{X}^{\varepsilon,R},U^\epsilon,
{W} )^{-1}$ 
on the space
\[
\tilde{\Omega}=C([0,T],H_d)\times C([0,T],H_d)
\times C([0,T],H_\infty)
\]
equipped with the Borel-$\sigma$-algebra $\mathcal{B}$,
and let $(\xi,\eta,\omega)$ denote the coordinate process on $\tilde{\Omega}$.

By tightness of 
$\{(\hat{X}^{\varepsilon,R},U^\epsilon
,{W})\}_{\varepsilon > 0}$,
there exists a subsequence $(\varepsilon_n)_{n\in\N}$ such that
$\PP^{R,\varepsilon_n}$ weakly converges to a probability measure $\PP^R$ 
on $(\tilde{\Omega},\mathcal{B})$, when $n\uparrow\infty$.

Let $\tilde{\mathcal{F}}$ be the $\PP^R$-\,completion of $\mathcal{B}$,
and let $(\tilde{\mathcal{F}}_t)_{t\in[0,T]}$ be the smallest filtration
the process $(\xi,\eta,\omega)$ is adapted to, on the one hand,
and which satisfies the usual conditions with respect to $\PP^R$, on the other.
Also, introduce $\tilde{\mathcal{F}}^n,\,(\tilde{\mathcal{F}}_t^n)_{t\in[0,T]}$
in a similar way with respect to $\PP^{R,\varepsilon_n},\,n\in\N$.

Now, it easily follows from \autoref{prop:conv} that,
on $(\tilde{\Omega},\tilde{\mathcal{F}},\PP^R)$, the following distributional properties must hold
for the pair of processes $(\eta,\omega)$:
$\eta$ is a $d$-dimensional Wiener process 
with covariance $(\sum_{\ell,m \in \N} b^i_{\ell,m}b^j_{\ell,m})_{i,j=1}^d$,
$\omega$ is a $Q$-Wiener process, %as described in \autoref{realTimeWipro},
$\eta$ and $\omega$ are independent.

Introduce
\begin{align*}
M^{R}_t = \xi_t - x_0 - \int_0^t \left( F_R(s,\xi_s) + C_R(s,\xi_s)\right) ds - \eta_t,
\quad t\in[0,T],\numberthis\label{martPro}
\end{align*}
and observe that each component of both processes $M^R$ and $\omega$, but also
\begin{gather*}
M^{R,i}_t M^{R,j}_t-
\int_0^t \sum_{m \in \N}
\sigma^{i,m}_R(s,\xi_s)
\sigma^{j,m}_R(s,\xi_s) q_m ds,
\quad t\in[0,T],\quad i,j=1,\dots,d,
\\
M^{R,i}_t \omega^m_t -
\int_0^t 
\sigma^{i,m}_R(s,\xi_s) q_m ds,
\quad t\in[0,T],\quad i=1,\dots,d,\quad m \in \mathbb{N},
\\
\omega^\ell_t \omega^m_t -
t \delta_{\ell,m} q_m, \quad t\in[0,T],\quad \ell, m \in \mathbb{N},
\end{gather*} 
are continuous local martingales
with respect to $(\tilde{\mathcal{F}}_t^n)_{t\in[0,T]}$
on $(\tilde{\Omega},\tilde{\mathcal{F}}^n,\PP^{R,\varepsilon_n})$, for any $n\in\N$,
and hence they are continuous local martingales
with respect to $(\tilde{\mathcal{F}}_t)_{t\in[0,T]}$
on $(\tilde{\Omega},\tilde{\mathcal{F}},\PP^{R})$, too,
by \cite[IX. Cor.1.19]{JS02}.

Therefore, applying \cite[Theorem 8.2]{DPZa14} to
the pair of process $(M^R,\omega)$ yields
\begin{align*}
M^R_t = \int_0^t \sigma_R(s,\xi_s ) d{W}^R_s,
\quad
\omega_t = \int_0^t 1\, d{W}^R_s = {W}^R_t,
\quad t\in[0,T], 
\end{align*} 
on $(\tilde{\Omega},\tilde{\mathcal{F}},\PP^R)$, 
or an enlargement of this space we still denote by $(\tilde{\Omega},\tilde{\mathcal{F}},\PP^R)$,
where $W^R$ is another $Q$-Wiener process,
which, by the above representation,
even $\PP^R$-\,almost surely coincides with $\omega$, so that
\begin{align*}
M^R_t = \int_0^t \sigma_R(s,\xi_s ) d\omega_s,\quad t\in[0,T],\quad\mbox{$\PP^R$-\,a.s.}
\end{align*} 

Thus, equation \eqref{martPro} can be written as
\begin{align*}
\xi_t 
\,=\,
x_0 + \int_0^t \left( F_R(s,\xi_s) + C_R(s,\xi_s)\right) ds + 
\int_0^t \sigma_R(s,\xi_s ) d\omega_s + \eta_t,
\quad t\in[0,T],\quad\mbox{$\PP^R$-\,a.s.},
\end{align*}
where $\omega$ is a $Q$-Wiener process, 
while $\eta$ is a $d$-dimensional Wiener process, independent of $\omega$,
and with covariance $(\sum_{\ell,m \in \N} b^i_{\ell,m}b^j_{\ell,m})_{i,j=1}^d$.
Observe that the process $\bar{X}^R$ satisfies the same type of equation, 
as $\sum_{\ell, m \in \N} b_{\ell,m} \bar{W}^{\ell,m}$ from \eqref{eq:X}
is a $d$-dimensional Wiener process 
with covariance $(\sum_{\ell,m \in \N} b^i_{\ell,m}b^j_{\ell,m})_{i,j=1}^d$, too.
But, since this type of equation admits a unique strong solution, the laws of $\xi$ and $\bar{X}^R$ 
must be the same, proving $\hat{X}^{\varepsilon_n,R}\to\bar{X}^R$, 
in law, when $n\uparrow\infty$. However, the same argument applies to any converging subsequence,
and the limit will always be the same, finally proving 
$\hat{X}^{\varepsilon,R}\to\bar{X}^R$, 
in law, when $\varepsilon\downarrow 0$.
\end{proof}

It remains to discuss how $R$ can be taken to infinity.

Recall that $\bar{X}$ is the solution of \eqref{eq:X},
and it is not difficult to see that $\bar{X}^R$ converges to $\bar{X}$, in law,  as $R \to \infty$.

Now take a function $\varphi_R \in C(C([0,T],\R^d),[0,1])$, 
such that $\varphi_R(u)=0$, if $\sup_{t \in [0,T]} |u_t| \leq R-1$, 
and $\varphi_R(u)=1$, if $\sup_{t \in [0,T]} |u_t| > R$. 

Then,
\begin{align*}
\PP \{ \tau^\varepsilon_R < T\} \leq 
\PP \ag \sup_{t \in [0,T]}
|\hat{X}^{\varepsilon,R}_t| \geq R 
\cg
\leq 
\E \left[ \varphi_R(\hat{X}^{\varepsilon,R})\right],
\end{align*}
and because $\hat{X}^{\varepsilon,R} \to \bar{X}^R$, in law, when $\varepsilon\downarrow 0$, we deduce that
\begin{align*}
\limsup_{\varepsilon \to 0}
\PP \{ \tau^\varepsilon_R < T\} 
\leq 
{\E} \left[ \varphi_R(\bar{X}^R)\right]
\leq
{\PP} \ag \sup_{t \in [0,T]}
|\bar{X}^R_t| \geq R-1 \cg
=
\PP \ag \sup_{t \in [0,T]}
|\bar{X}_t| \geq R-1 \cg,
\end{align*}
where the last probability converges to zero, when $R\to\infty$, because $\bar{X}$ is a global solution.

As a consequence, for any $\psi \in C_b(C([0,T],\R^d),\R)$,
\begin{align*}
\left| \E \left[\psi(X^\varepsilon)\right] - \E \left[\psi(\bar{X})\right] \right|
\leq\,
&\left| \E \left[\psi(X^\varepsilon)\right] - \E \left[\psi(X^\varepsilon_{\cdot \wedge \tau^\varepsilon_R})\right] \right|
+
\left| \E \left[\psi(X^\varepsilon_{\cdot \wedge \tau^\varepsilon_R})\right] - \E \left[\psi(\hat{X}^{\varepsilon,R}_{\cdot \wedge \tau^\varepsilon_R})\right] \right| \\
&+
\left| \E \left[\psi(\hat{X}^{\varepsilon,R}_{\cdot \wedge \tau^\varepsilon_R})\right] - \E \left[\psi(\hat{X}^{\varepsilon,R})\right] \right| 
+
\left| \E \left[\psi(\hat{X}^{\varepsilon,R})\right] - {\E} \left[\psi(\bar{X}^R)\right] \right| \\
&+
\left| {\E} \left[\psi(\bar{X}^R)\right] - \E \left[\psi(\bar{X})\right] \right|.
\end{align*}
Here, when
taking $R$ large enough, we can make all the summands on the right-hand side,
except for the second and fourth, arbitrarily small, uniformly in $\varepsilon$,
and, for fixed $R$, the remaining terms go to zero, when $\varepsilon \downarrow 0$.

Thus, by a diagonal argument, the convergence in law of $X^\varepsilon \to \bar{X},\,\varepsilon\downarrow 0$, follows,
completing the proof of the theorem.
\section{Application to Climate Models} \label{sec:clim}
We now apply \autoref{thm:main} to perform stochastic model reduction 
for a subclass of the stochastic climate models
given by \eqref{eq:introXeps}, \eqref{eq:introYeps} in the introduction:
we restrict ourselves to a simpler version of \eqref{eq:introYeps},
omitting fast forcing $\varepsilon^{-2}f^2_{\varepsilon^{-1}t}$ and $\varepsilon^{-1} A^2_2 Y^\varepsilon_t$, on the one hand,
but also neglecting the interaction $B^2_{12}(X^\varepsilon_t,Y^\varepsilon_t)$, on the other.
While the first two terms we omit are technically demanding but look doable from a wider prospective, 
which is beyond this paper, the term $\varepsilon^{-1}B^2_{12}(X^\varepsilon_t,Y^\varepsilon_t)$
involving the neglected interaction is notoriously hard and beyond our understanding, right now.

For each $\varepsilon>0$, 
let $(X^\varepsilon,Y^\varepsilon)$ be a pair of processes satisfying
\begin{align}
\frac{dX^\varepsilon_t}{dt} &= F^1_t +  A^1_1 X^\varepsilon_t +  A^1_2 Y^\varepsilon_t 
+ B^1_{11}(X^\varepsilon_t,X^\varepsilon_t) + B^1_{12}(X^\varepsilon_t,Y^\varepsilon_t)
+ \varepsilon B^1_{22}(Y^\varepsilon_t,Y^\varepsilon_t), \label{eq:Xepsclim}\\
\frac{dY^\varepsilon_t}{dt} &= \varepsilon^{-2} A^2_1 X^\varepsilon_t 
%{\color{red} \,+\,\varepsilon^{-1} A^2_2 Y^\varepsilon_t }
+\varepsilon^{-2} B^2_{11}(X^\varepsilon_t,X^\varepsilon_t) -\varepsilon^{-2} Y^\varepsilon_t 
+  \varepsilon^{-2} \dot{W}_t, \label{eq:Yepsclim}
\end{align}
where $A^1_1:H_d \to H_d$, $A^1_2:H_\infty \to H_d$, $A^2_1:H_d \to H_\infty$ are bounded linear operators, $B^1_{11}:H_d \times H_d \to H_d$, $B^1_{12}:H_d \times H_\infty \to H_d$,
$B^1_{22}:H_\infty \times H_\infty \to H_d$, $B^2_{11}:H_d \times H_d \to H_\infty$ 
are continuous bilinear maps, 
and $F^1:[0,T] \to H_d$ is a deterministic continuous external force. 
Stochastic basis and Wiener process $W$ are taken to be the same as in \autoref{realTimeWipro}.

In what follows, the above equations will always have initial conditions $(x_0,y_0)$, 
where $x_0 \in H_d$ can be chosen arbitrarily, 
while $y_0 = \int_{-\infty}^0 \varepsilon^{-2} e^{\varepsilon^{-2}s} dW_s$
will be fixed to ensure pseudo stationarity of the scaled unresolved variables.
Note that fixing $y_0\in H_\infty$ this way would not restrict the initial data of the reduced equations.

In fluid dynamics settings like \eqref{eq:introZ},
it is customary to assume that $A$ is self-adjoint, 
and that the full nonlinearity is skew-symmetric: 
$\langle B(z',z),z \rangle_H = 0$,  $z,z' \in H$, see \cite{MaWa06}.
We therefore make the following assumptions on the projected coefficients:
\begin{itemize}
\item[(\textbf{C1})] $A^2_1 = (A^1_2)^*$;
\item[(\textbf{C2})] $\langle B^1_{11}(x',x),x \rangle_{H_d} = 0$, for all $x,x' \in H_d$;
\item[(\textbf{C3})] $\langle B^1_{12}(x',y),x \rangle_{H_d} = - \langle B^2_{11}(x',x),y \rangle_{H_\infty} $, 
for all $x,x' \in H_d$, $y \in H_\infty$.
\end{itemize}

Also, without loss of generality, 
we can assume that $B^1_{22}$ is symmetric in the sense of
$\langle B^1_{22}(\mathbf{f}_\ell,\mathbf{f}_m) , \mathbf{e}_i \rangle_{H_d}
=
\langle B^1_{22}(\mathbf{f}_m,\mathbf{f}_\ell) , \mathbf{e}_i \rangle_{H_d}$,
for all $i,\ell,m$; and finally we will need the analogue of (A5), that is
\begin{itemize}
\item[(\textbf{C4})] 
$\sum_{\ell \in \N} 
 \langle B^1_{22}(\mathbf{f}_\ell,\mathbf{f}_\ell) , \mathbf{e}_i \rangle_{H_d}\, q_\ell\,=0$,
 for all $i=1,\dots,d$.
\end{itemize}
Note that the latter condition is indeed satisfied for many fluid-dynamics models---it usually holds 
independently of the structure of the noise because 
$\langle B^1_{22}(\mathbf{f}_\ell,\mathbf{f}_m) , \mathbf{e}_i \rangle_{H_d}$
would be zero on the diagonal, when $\ell=m$, for all $i$.

Next, we bring equations \eqref{eq:Xepsclim},\eqref{eq:Yepsclim} into a form
which makes them  comparable to \eqref{eq:introabsX},\eqref{eq:introabsY}.

Using the definition of $y_0$,
we have the following mild formulation of \eqref{eq:Yepsclim},
\begin{equation} \label{eq:Yepsclim'}
Y_t^\varepsilon\,=\,
\tilde{Y}^\varepsilon_t + 
\int_0^t \varepsilon^{-2} e^{-\varepsilon^{-2}(t-s)}
\left( A^2_1 X^\varepsilon_s 
+B^2_{11}(X^\varepsilon_s,X^\varepsilon_s)\right) ds,
\quad t\in[0,T],
\end{equation}
where
\begin{equation*}
\tilde{Y}^\varepsilon_t = \int_{-\infty}^t \varepsilon^{-2}e^{-\varepsilon^{-2}(t-s)} dW_s,
\quad t\in\R,
\end{equation*}
is a stationary Ornstein-Uhlenbeck process.
Plugging \eqref{eq:Yepsclim'} into \eqref{eq:Xepsclim},
$X^\varepsilon$ alternatively satisfies
\begin{align*}
X^\varepsilon_t = \,
&\,x_0 + \int_0^t \left( F^1_s +  A^1_1 X^\varepsilon_s + B^1_{11}(X^\varepsilon_s,X^\varepsilon_s) \right) ds 
+ \int_0^t A^1_2 Z^\varepsilon_s ds
+ \int_0^t B^1_{12}\left(X^\varepsilon_s,Z^\varepsilon_s\right) ds
\numberthis \label{eq:Xepsclim1'}\\
&+ \int_0^t A^1_2 \tilde{Y}^\varepsilon_s ds 
+ \int_0^t B^1_{12}(X^\varepsilon_s,\tilde{Y}^\varepsilon_s) ds\\
&+\int_0^t \varepsilon B^1_{22}(\tilde{Y}^\varepsilon_s,\tilde{Y}^\varepsilon_s) ds 
%+\int_0^t \varepsilon B^1_{22}(Z^\varepsilon_s,\tilde{Y}^\varepsilon_s) ds 
+2\int_0^t \varepsilon B^1_{22}(\tilde{Y}^\varepsilon_s,Z^\varepsilon_s ) ds 
+\int_0^t \varepsilon B^1_{22}\left(Z^\varepsilon_s,Z^\varepsilon_s\right) ds,
\quad t\in[0,T],
\end{align*}
when using the abbreviation
\begin{align*}
Z^\varepsilon_s =
\int_0^s \varepsilon^{-2} e^{-\varepsilon^{-2}(s-r)}
\left( A^2_1 X^\varepsilon_r 
+B^2_{11}(X^\varepsilon_r,X^\varepsilon_r)\right) dr.
\end{align*}

Since $Z^\varepsilon_s$ is close to 
$A^2_1 X^\varepsilon_s + B^2_{11}(X^\varepsilon_s,X^\varepsilon_s)$,
for small $\varepsilon$,
and since both terms
$B^1_{22}(\tilde{Y}^\varepsilon_s,Z^\varepsilon_s ),\,
B^1_{22}\left(Z^\varepsilon_s,Z^\varepsilon_s\right)$
will be shown to vanish with $\varepsilon$, too,
the process $X^\varepsilon$ should be close to $\tilde{X}^\varepsilon$ satisfying
\begin{align*}
\tilde{X}^\varepsilon_t = \,
&\,x_0 
+ \int_0^t \left( F^1_s +  A^1_1 \tilde{X}^\varepsilon_s 
+ B^1_{11}(\tilde{X}^\varepsilon_s,\tilde{X}^\varepsilon_s) \right) ds 
+ \int_0^t A^1_2 \left( A^2_1 \tilde{X}^\varepsilon_s 
+B^2_{11}(\tilde{X}^\varepsilon_s,\tilde{X}^\varepsilon_s) \right) ds
\numberthis \label{eq:XepsTilde}\\
&+ \int_0^t B^1_{12}\left(\tilde{X}^\varepsilon_s,
\left( A^2_1 \tilde{X}^\varepsilon_s 
+B^2_{11}(\tilde{X}^\varepsilon_s,\tilde{X}^\varepsilon_s)\right) \right)ds \\
&+ \int_0^t A^1_2 \tilde{Y}^\varepsilon_s ds 
+ \int_0^t B^1_{12}(\tilde{X}^\varepsilon_s,\tilde{Y}^\varepsilon_s)\,ds
+ \int_0^t \varepsilon B^1_{22}(\tilde{Y}^\varepsilon_s,\tilde{Y}^\varepsilon_s)\,ds,
\quad t\in[0,T],
\end{align*} 
which is an equation of type \eqref{eq:introabsX} with
\begin{align*}
F(t,x)\,=\,
&\,F^1_t
+ A^1_1{x} + B^1_{11}(x,x) + A^1_2 \left( A^2_1{x} + B^2_{11}(x,x) \right)
+ B^1_{12}\left(x,\left( A^2_1{x} + B^2_{11}(x,x)\right) \right),\\
\sigma(t,x)\,=\,
&\,A^1_2 + B^1_{12}(x,\cdot)\,,\\
\beta\,=\,&\,B^1_{22}\,.
\end{align*}

Thus, in this setting, the analogue of \eqref{eq:X} would read
\begin{align*} 
\bar{X}_t = \,&x_0 
+ \int_0^t \left( F^1_s 
+  A^1_1 \bar{X}_s + B^1_{11}(\bar{X}_s,\bar{X}_s) \right) ds \numberthis \label{eq:Xclim}
+ \int_0^t A^1_2 \left( A^2_1 \bar{X}_s + B^2_{11}(\bar{X}_s,\bar{X}_s) \right) ds \\
&+ \int_0^t B^1_{12}\left(\bar{X}_s,
\left(A^2_1 \bar{X}_s + B^2_{11}(\bar{X}_s,\bar{X}_s)\right) \right)ds
+ \int_0^t C(\bar{X}_s)\,ds\\
&+ A^1_2 W_t
+ \int_0^t B^1_{12}(\bar{X}_s,dW_s) 
+ \sum_{\ell, m \in \N} b_{\ell,m} \bar{W}^{\ell,m}_t,
\quad t\in[0,T],
\end{align*} 
where the Stratonovich correction term $C: H_d \to H_d$ simplifies to
\begin{equation*}
\langle C({x}) ,   \mathbf{e}_i \rangle_{H_d}
= \frac{1}{2} \sum_{m \in \N} q_m \sum_{j=1}^d
\langle B^1_{12}(\mathbf{e}_j,\mathbf{f}_m), \mathbf{e}_i \rangle_{H_d}
\langle B^1_{12}(x,\mathbf{f}_m), \mathbf{e}_j \rangle_{H_d},
\quad
i=1,\dots,d,
\end{equation*}
and
\begin{equation*}
b^i_{\ell,m} = \langle B^1_{22}(\mathbf{f}_\ell,\mathbf{f}_m) , \mathbf{e}_i \rangle_{H_d}\,
\sqrt{\frac{q_\ell q_m}{2}},
\quad  i=1,\dots,d,\;\ell,m\in\N.
\end{equation*}
\begin{prop} \label{prop:wellpos}
When assuming (C1)-(C3), equation \eqref{eq:Xclim} admits a unique global strong solution on $[0,T]$.
\end{prop}
\begin{proof}
First,
regularity of coefficients guarantees the existence of a unique local strong solution.
Second, by It\^o's formula,
\begin{align*}
\frac{1}{2}|\bar{X}_{t\wedge\tau}|^2 \,=\, &\,\frac{1}{2}|x_0|^2 
+ \int_0^{t\wedge\tau} \langle F^1_s +  A^1_1 \bar{X}_s + B^1_{11}(\bar{X}_s,\bar{X}_s) , \bar{X}_s \rangle\,ds \\
&+ \int_0^{t\wedge\tau} \langle  A^1_2 \left( A^2_1 \bar{X}_s 
+B^2_{11}(\bar{X}_s,\bar{X}_s) \right), \bar{X}_s \rangle\,ds\\
&+  \int_0^{t\wedge\tau} \langle B^1_{12}\left(\bar{X}_s,
\left( A^2_1 \bar{X}_s 
+B^2_{11}(\bar{X}_s,\bar{X}_s)\right) \right) , \bar{X}_s \rangle\,ds\,
+ \int_0^{t\wedge\tau}\langle C(\bar{X}_s),\bar{X}_s \rangle\,ds\\
&+ \int_0^{t\wedge\tau} \langle A^1_2 dW_s  , \bar{X}_s \rangle  
+ \int_0^{t\wedge\tau} \langle B^1_{12}(\bar{X}_s,dW_s), \bar{X}_s \rangle 
+ \sum_{\ell,m\in\N} \int_0^{t\wedge\tau}\langle b_{\ell,m},\bar{X}_s \rangle\,d\bar{W}^{\ell,m}_s\\
&+\frac{1}{2}\sum_{m\in\N} |A^1_2\mathbf{f}_m|^2 q_m (t\wedge\tau)
+\frac{1}{2}\sum_{m\in\N} \int_0^{t\wedge\tau}\!\! |B^1_{12}(\bar{X}_s,\mathbf{f}_m)|^2 q_m\,ds
+\frac{1}{2}\sum_{\ell,m\in\N} |b_{\ell,m}|^2 (t\wedge\tau),
\end{align*}
for any fixed $t\in[0,T]$, and any stopping time $\tau$ smaller than a possible explosion time.

Applying (C1)-(C3), we have the identities
\begin{gather*}
\langle B^1_{11}(\bar{X}_s,\bar{X}_s) , \bar{X}_s \rangle_{H_d} = 0, \\
\langle A^1_2  B^2_{11}(\bar{X}_s,\bar{X}_s), \bar{X}_s \rangle_{H_d}=
\langle B^2_{11}(\bar{X}_s,\bar{X}_s),A^2_1 \bar{X}_s \rangle_{H_\infty}, \\
\langle B^1_{12}(\bar{X}_s,A^2_1 \bar{X}_s) , \bar{X}_s \rangle_{H_d} =
-\langle B^2_{11}(\bar{X}_s,\bar{X}_s) , A^2_1 \bar{X}_s \rangle_{H_\infty},\\
\langle B^1_{12}(\bar{X}_s,B^2_{11}(\bar{X}_s,\bar{X}_s)) , \bar{X}_s \rangle_{H_d} =
-\| B^2_{11}(\bar{X}_s,\bar{X}_s)\|^2_{H_\infty},
\end{gather*}
leading to
\begin{equation*}
\E \left[ \sup_{t'\leq t}|\bar{X}_{t'\wedge\tau}|^2 \right] 
\lesssim \left( 1 + \int_0^t \E \left[ \sup_{s'\leq s}|\bar{X}_{s'\wedge \tau}|^2 \right] ds\right),
\end{equation*}
again using the regularity of the coefficients combined with Burkholder-Davis-Gundy's inequality.
Thus, by Gronwall, the local solution $\bar{X}$ has to be global on $[0,T]$.
\end{proof}
\begin{rmk}\label{allFine}\rm
In a very similar way, it can be shown that both equations
\eqref{eq:Xepsclim1'} \& \eqref{eq:XepsTilde}
admit unique global strong solutions on $[0,T]$, too,
and hence those proofs are omitted.
As a consequence, 
simply substituting the solution of \eqref{eq:Xepsclim1'} into \eqref{eq:Yepsclim'},
for each $\varepsilon>0$,
there is a unique pair of processes $(X^\varepsilon,Y^\varepsilon)$
satisfying \eqref{eq:Xepsclim},\eqref{eq:Yepsclim} on $[0,T]$.
\end{rmk}
\begin{thm} \label{thm:climate}
Assume (C1)-(C3), fix $\varepsilon>0$,
and let $(X^\varepsilon,Y^\varepsilon)$ be the unique pair of processes
satisfying \eqref{eq:Xepsclim},\eqref{eq:Yepsclim}
on a given climate time interval $[0,T]$.

(i)
If (C4), then $X^\varepsilon$ converges in law, $\varepsilon\downarrow 0$,
to the unique process $\bar{X}$ satisfying \eqref{eq:Xclim}.

(ii)
However, if (C4) comes via $B^1_{22}=0$, then the stronger convergence \eqref{eq:introlim} holds true.
\end{thm} 
\begin{proof}
Recall the process $\tilde{X}^\varepsilon$ satisfying \eqref{eq:XepsTilde},
which is an equation of type \eqref{eq:introabsX} with coefficients $F,\sigma,\beta$ satisfying (A1)-(A3). 
Furthermore, by \autoref{prop:wellpos} and \autoref{allFine}, condition (A4) is satisfied, too,
while (A5) and (C4) actually are the same condition.

All in all, \autoref{thm:main} implies that both parts (i) \&\ (ii) of \autoref{thm:climate} hold true
when replacing ${X}^\varepsilon$ by $\tilde{X}^\varepsilon$.
 
Thus, it is sufficient to prove convergence in probability of $X^\varepsilon - \tilde{X}^\varepsilon$ to zero, 
$\varepsilon \downarrow 0$, uniformly on compact subsets of a localising stochastic interval,
which can easily be shown following the lines of proof of \autoref{thm:main}. 

Indeed, by localization and discretization arguments, one would first derive
\[
\E \left[ \sup_{\substack{k'=0,\dots,h\\k'\!\Delta \leq \tau_R^\varepsilon}}
\left| X^\varepsilon_{k'\Delta} - \tilde{X}^\varepsilon_{k'\Delta} \right|^2 \right] 
\lesssim \;r(\Delta,\varepsilon) +
\sum_{k=0}^{h-1} 
\Delta \E \left[ \sup_{\substack{k'=0,\dots,k\\k'\!\Delta \leq \tau_R^\varepsilon}}
\left| X^\varepsilon_{k'\Delta} - \tilde{X}^\varepsilon_{k'\Delta} \right|^2 \right],
\quad h=1,\dots,[T/\Delta],
\]
where 
$\tau^\varepsilon_R = \inf\{t \geq 0: |X^\varepsilon_t| \geq R \} \wedge \inf\{t \geq 0: |\tilde{X}^\varepsilon_t| \geq R \}$,
and $r(\Delta,\varepsilon) \to 0,\,\varepsilon \downarrow 0$, 
for a suitable choice of $\Delta=\Delta_\varepsilon$.
Then, combining Gronwall's lemma and Markov's inequality, one would obtain
\begin{equation*}
\lim_{\varepsilon \to 0} \PP\ag \sup_{t \leq T \wedge \tau^\varepsilon_R }\|X^\varepsilon_t - \tilde{X}^\varepsilon_t \|_{H_d} > \delta\cg = 0,\quad\forall\,\delta>0,
\end{equation*} 
which yields the convergences stated in parts (i) and (ii) of \autoref{thm:climate} up to time $\tau^\varepsilon_R$.
Since $\bar{X}$ is globally defined, both types of convergence can be extended to the whole interval $[0,T]$,
using similar arguments given in the proof of the corresponding parts of \autoref{thm:main}.
\end{proof}

\end{document}